\documentclass[a4paper,11pt,fleqn]{article}

\usepackage[utf8]{inputenc}
\usepackage[T1]{fontenc}
\usepackage{textcomp}
\usepackage{amsmath,amssymb,amsthm,mathrsfs}
\usepackage{lmodern}
\usepackage[a4paper]{geometry}
\usepackage{graphicx}
\usepackage[usenames,dvipsnames]{xcolor}
\usepackage{microtype}
\usepackage{enumerate}

\usepackage{hyperref}
\hypersetup{pdfstartview=XYZ}
\newtheoremstyle{theoreme}
 {\topsep} 
 {\topsep} 
 {} 
 {0pt} 
 {\bfseries} 
 {\newline} 
 {3pt} 
 {} 
\newtheoremstyle{proposition}
{\topsep} 
 {\topsep} 
 {} 
 {0pt} 
 {\bfseries} 
 {\newline} 
 {3pt} 
 {} 
 \newtheoremstyle{lemma}
 {\topsep} 
 {\topsep} 
 {} 
 {0pt} 
 {\bfseries} 
 {\newline} 
 {3pt} 
 {} 
\newtheoremstyle{definition}
 {\topsep}
 {\topsep}
 {}
 {0pt}
 {\bfseries}
 {\newline}
 {3pt}
 {}
\newtheoremstyle{remarque}
 {\topsep}
 {\topsep}
 {}
 {0pt}
 {\bfseries}
 {\newline}
 {3pt}
 {} 
\theoremstyle{theoreme}
\newtheorem{Thm}{Theorem}[section]
\theoremstyle{lemma}
\newtheorem{Lem}[Thm]{Lemma}
\theoremstyle{proposition}
\newtheorem{Prop}[Thm]{Proposition}

\theoremstyle{definition}
\newtheorem{Def}[Thm]{Definition}
\theoremstyle{remarque}
\newtheorem{Rem}[Thm]{Remark}
\newtheorem{Ex}[Thm]{Example}
\newtheorem{Exs}[Thm]{Examples}
\newtheorem*{Notation}{Notation}

\newcommand\entiers{\mathbb{N}}

\newcommand\entiersrel{\mathbb{Z}}

\newcommand\N{\mathbb{N}}

\newcommand\Z{\mathbb{Z}}

\newcommand\chain{\mathcal{C}}

\newcommand\etoile{\precsim^{\ast}}
\newcommand\netoile{\precnsim^{\ast}}

\newcommand\ig{g^{-1}}
\newcommand\ih{h^{-1}}
\newcommand\ix{x^{-1}}
\newcommand\iy{y^{-1}}
\newcommand\iz{z^{-1}}
\newcommand\iu{u^{-1}}
\newcommand\iv{v^{-1}}
\newcommand\comment[1]{}
\newcommand\T{\mathfrak{T}}
\begin{document}
 \title{On the structure of groups endowed with a compatible C-relation\footnote{All results presented here are part of my PhD. 
   In that regard I thank my supervisors Salma Kuhlmann and Françoise Point for 
   the help and support they gave me 
   during its completion. I am particularly grateful to Salma Kuhlmann for taking the time to 
   read and discuss my paper, as her many comments and suggestions greatly improved it.
   I also thank Françoise Delon for her questions which motivated me to investigate the connections between 
   C-relations and quasi-orders. I would also like to thank the referee for their careful reading of my paper and for their useful comments.}}
 \author{Gabriel Lehéricy}
 \maketitle
 
 \begin{abstract}
  We use quasi-orders to describe the structure of C-groups. We do this by associating a quasi-order to each 
  compatible $C$-relation of a group, and then give the structure of such quasi-ordered groups. We also 
  reformulate in terms of quasi-orders some results concerning C-minimal groups given in \cite{Macstein}.\\
  
  
 \end{abstract}

 \section*{Introduction}
 
 The notion of C-relation  was first introduced by Adeleke and Neumann in \cite{Adeleke1}, where it was used to study 
 certain groups of automorphisms called
 Jordan groups. In \cite{Macstein}, Macpherson and Steinhorn introduced the notion of C-group and C-minimal structure and
 gave a partial description of C-minimal groups. The notion of C-relation defined in \cite{Macstein} is now called a dense C-relation 
 (see the definition in Section \ref{Cminsection} below). A more general notion of C-relation was introduced by
 Delon  in \cite{Delon} (see Delon's definition of a C-relation in Section  \ref{preliminaries} below).
 In Delon's context, o-minimality and strong minimality both become 
 special cases of C-minimality. 
 Until now, all the work concerning C-groups (see for example \cite{Macstein},\cite{Simonetta1} and 
 \cite{Simonetta2}) has focused on the study of C-minimal groups.
 The main motivation behind this paper is to understand the structure 
 of an arbitrary C-group, i.e without any assumption of minimality. We then apply our general theory to the 
 special case of dense C-minimal groups in the last section of the paper.
 
 We already know two examples of C-groups: those whose C-relation comes from an order and those whose C-relation comes 
   from a valuation. The goal of this paper is to show that these two fundamental examples are the 
   ``building blocks'' of the class of C-groups, in the sense that any compatible C-relation on a group can be constructed 
   from C-relations induced by valuations and C-relations induced by orders. 
   This is achieved not by working directly with a C-relation but with a quasi-order canonically associated to the
   C-relation, which we call a C-quasi-order (abbreviated as C-q.o).
  
  Except for Section \ref{compatiblesection}, which is not essential to understand the main results of this paper, 
  all results presented here are independent from our
  work   on compatible quasi-orders done in \cite{Lehericy}. However, 
  the main ideas behind the method used in the current paper are greatly inspired by what we did in 
  \cite{Lehericy}, which is why we would like to briefly recall the important results
  of \cite{Lehericy}.  
  We defined a compatible 
  quasi-ordered abelian group (q.o.a.g) as a pair $(G,\precsim)$ where $G$ is an abelian group and $\precsim$ 
  a compatible quasi-order, i.e a quasi-order satisfying the following axioms (see Section \ref{preliminaries} below for the definition of 
  ``$\sim$''):
  \begin{itemize}
    \item[$(Q_1)$] $\forall x\quad (x\sim0\Rightarrow x=0)$.
    \item[$(Q_2)$] $\forall x,y,z \quad(x\precsim y\nsim z\Rightarrow x+z\precsim y+z)$.
  \end{itemize}
 Fixing a compatible q.o.a.g $(G,\precsim)$, we distinguished two kinds of elements in $G$ respectively called 
 o-type and v-type elements. The v-type elements are characterized by the fact 
 that they are equivalent to their inverse, whereas o-type elements are not.
 We showed that the set $G^o$ of o-type elements of $G$ is a subgroup of $G$ and that $\precsim$ is actually an order on $G^o$, whereas 
 $\precsim$ is valuational on the set $G^v$ of v-type elements.  We also showed that any 
 compatible quasi-order naturally induces a compatible C-relation.  It was however quickly established that some compatible
 C-relations are not induced by a compatible quasi-order,
   so that the notion of compatible quasi-order was not appropriate to describe the entire class of compatible
   C-relations. This is what lead us to develop the notion of C-quasi-order introduced in this paper.

  \comment{It was already showed in \cite{Lehericy} that quasi-orders are connected to C-relations: more precisely, 
   it was showed there that any compatible quasi-order (as defined in \cite{Lehericy}) naturally induces a C-relation.
   It was however quickly established that some compatible C-relations are not induced by a compatible quasi-order,
   so that the notion of compatible quasi-order was not appropriate to describe the entire class of compatible
   C-relations. This lead us to consider another kind of quasi-order, called C-quasi-orders, which we define in 
   this paper and then use to describe the structure of an arbitrary C-group.}
   
   C-quasi-orders are quasi-orders canonically induced by a compatible C-relation on a group. 
   Since there is a bijective 
   correspondence between compatible C-relations and C-quasi-orders, we can study the class of C-groups by studying  the class 
   of C-quasi-orders, and this is what we do in this paper.  Taking a group $G$ with 
   a C-quasi-order $\precsim$, we show that $\precsim$ is basically a mix 
    of valuational quasi-orders with C-quasi-orders induced by group orderings.  The main 
    idea is to distinguish two kinds of elements, respectively called o-type and v-type (analogously to what was done in 
 \cite{Lehericy}) and to associate to each 
 $g\in G$ a subset $T_g$ of $G$ called the type-component of $G$. This set $T_g$ is characterized by two properties 
 : $T_g$ is strictly convex, and if $g$ is v-type (respectively o-type), then the C-quasi-order $\precsim$ is
 valuational-like (respectively, order-type-like) on $T_g$ (see Remark \ref{lookslikeexplain} for the definitions of ``valuation-like'' and ``order-type-like''). Moreover, $T_g$ is maximal with these 
 properties. We can then show that the family of all type-components form a partition of $G$. 
 
 We also draw attention to a counter-intuitive phenomenon, which we call welding,
 which occurs in certain C-quasi-ordered groups. Welding happens when the group contains an o-type element which is equivalent
 to a v-type element. This is counter-intuitive, since one would expect the quasi-order to separate elements of different types.
 If there is no welding in the group, then the $T_g$'s are actually convex. However, if there 
 is welding at a point $g$, then the maximum of $T_g$ is equivalent to the minimum of a $T_h$, which means 
 that the type-components are only strictly convex. This also means that a C-q.o cannot in general be obtained by lifting 
 C-q.o's of elementary type (i.e C-q.o's induced by a valuation or an ordering). However, we will show (see Theorem \ref{structuretheorem}) 
 that any C-q.o can be obtained by first lifting  C-q.o's of elementary type
  and then 
 ``welding''  (see Proposition \ref{weldingconstruction}), i.e coarsening the quasi-order in a certain way.
   
   The first section gives preliminaries on C-relations and quasi-orders. In Section \ref{Cqosection} we introduce 
   C-quasi-orders.
   We then give an axiomatization of the class of C-quasi-orders and 
   describe the structure of a C-quasi-order induced by a group ordering.
    Section \ref{structuresection} is dedicated to the study of an arbitrary C-quasi-ordered group $(G,\precsim)$.
     We start by giving five examples of C-q.o's. In Section \ref{relationsection} we give some results 
    describing the relation between a C-q.o and the group operation, which will be essential in proving 
    the main results of Section \ref{typecomponents}. Section \ref{quotientsection} shows that $\precsim$ induces a C-quasi-order on any quotient 
    $G/H$ where $H$ is a strictly convex normal subgroup. In Section \ref{typecomponents} we define the type-component $T_g$ of an element $g$ and 
    describe its properties.  We also associate to $g$ two subgroups $G^g$ and $G_g$ of $G$ and show that 
    the C-q.o induced by $\precsim$ on the quotient $G^g/G_g$ comes from a valuation (respectively, from
    an ordering) if $g$ is v-type (respectively, if $g$ is o-type). We start Section \ref{maintheoremsection} by giving two 
    ways of constructing C-q.o's: lifting and welding. We then give our main result, Theorem \ref{structuretheorem},
    which states that any C-q.o can be obtained from C-q.o's of elementary type by lifting and welding.
     Finally, in Section \ref{Cminsection}, we reinterpret the results on dense C-minimal groups given in \cite{Macstein} in view 
     of our main theorem \ref{structuretheorem}. More precisely, we show that the assumption of C-minimality 
     imposes conditions on the type-components.
 
 \section{Preliminaries}\label{preliminaries}
     In this paper, $\N$ denotes the set of natural numbers $\{1,2,3,\dots\}$ without zero. The set 
     $\N\cup\{0\}$ is denoted by $\N_0$.
     An \textbf{ordered group} is a pair $(G,\leq)$ consisting of a group $G$ with a total order $\leq$ satisfying:
     \begin{equation}\label{axiomoagr}\tag{OG}
       \forall x,y,z\in G, x\leq y\Rightarrow xz\leq yz\wedge zx\leq zy.
      \end{equation}
      For any group $G$ and $g,z\in G$, $g^z$ denotes $zg\iz$.
      A \textbf{valuation} on a group $G$ is a map $v:G\to \Gamma\cup\{\infty\}$ such that:
	    \begin{enumerate}[(i)]
	    \item $\Gamma$ is a totally ordered set, and this order is extended to $\Gamma\cup\{\infty\}$ by declaring
	    $\gamma<\infty$ for all $\gamma\in\Gamma$.
	    \item For any $g\in G$, $v(g)=\infty\Leftrightarrow g=1$.
	    \item For any $g,h\in G$, $v(g\ih)\geq \min(v(g),v(h))$.
	    \item For any $g,h,z\in G$, $v(g)\leq v(h)\Leftrightarrow v(g^z)\leq v(h^z)$.
	    \end{enumerate}

  If $v:G\to\Gamma\cup\{\infty\}$ is a valuation, then, for any $\gamma\in \Gamma$, $G^{\gamma}$ and 
 $G_{\gamma}$ respectively denote $\{g\in G\mid v(g)\geq\gamma\}$ and $\{g\in G\mid v(g)>\gamma\}$.
  \begin{Rem}\label{propval}
 Note that due to the fact that $(g^z)^{z^{-1}}=g$, we can replace ``$\Leftrightarrow$'' by ``$\Rightarrow$''  in  (iv). 
Also, assuming that (ii) holds, one easily sees that (iii) holds if and only if for any $g,h\in G$, $v(g)=v(g^{-1})\wedge v(gh)\geq \min(v(g),v(h))$ holds.
Moreover, we can easily show that following facts are true for any valued group $(G,v)$:   
  \begin{enumerate}[(a)] 
  \item For any $g,h\in G$, $v(g)<v(h)\Rightarrow v(g^z)<v(h^z)$ and $v(g)=v(h)\Rightarrow v(g^z)=v(h^z)$ (it follows from (iv)).  
  \item  If $v(g)<v(h)$, then $v(gh)=v(g)=v(hg)$.
   \item For any
 $\gamma\in\Gamma$, $G_{\gamma}$ is a normal subgroup of $G^{\gamma}$. Note however that  it can happen that
 $v(g)\neq v(g^z)$, and in particular $G^{\gamma}$ and $G_{\gamma}$ are not always normal in $G$. This is showed by Example \ref{exempletypeval}.
 \item Thanks to 
 axiom (iv) of valuations, conjugation by an element $z\in G$ induces an automorphism of $\Gamma$ defined by 
 $v(g)\mapsto v(g^z)$ (note that this map is onto since $v(g^{\iz})$ is a pre-image of $v(g)$). If $\gamma=v(g)$, then we denote $v(g^z)$ by $\gamma^z$. Conjugation by $z$ also induces a 
 group homomorphism $G^{\gamma}\to G^{\gamma^z}$ and another one from $G^{\gamma}/G_{\gamma}$ to $G^{\gamma^z}/G_{\gamma^z}$.
 
  \end{enumerate}
 \end{Rem}

  A \textbf{C-relation} on a set $M$ (see \cite{Delon}) is a ternary relation $C$ satisfying the universal closure of 
  the following axioms:
	  \begin{itemize}
	\item[$(C_1)$] $C(x,y,z)\Rightarrow C(x,z,y)$.
	\item[$(C_2)$] $C(x,y,z)\Rightarrow\neg C(y,x,z)$.
	\item[$(C_3)$] $C(x,y,z)\Rightarrow C(w,y,z)\vee C(x,w,z)$.
	\item[$(C_4)$] $x\neq y\Rightarrow C(x,y,y)$.
	\end{itemize}

	 Note that $(C_2)$ implies $\neg C(x,x,x)$ for all $x$. 
If $G$ is a group and $C$ a C-relation on $G$, then we say that $C$ is \textbf{compatible} (with the group operation) if 
  $C(x,y,z)$ implies $C(vxu,vyu,vzu)$ for any $x,y,z,u,v\in G$. A \textbf{C-group} is a 
  pair $(G,C)$ consisting of a group $G$ with a compatible C-relation $C$.
 \begin{Ex}\label{exempleCgroupe}
  There are two fundamental examples of C-groups:
  \begin{enumerate}[(a)]
         \item If $(G,\leq)$ is a totally ordered group, then 
       $\leq$ induces a compatible C-relation defined by 
       $C(x,y,z)\Leftrightarrow (y<x\wedge z<x)\vee(y=z\neq x)$. Such a C-relation is called an \textbf{order-type} C-relation.
       \item If $(G,v)$ is a valued group, then $v$ induces a compatible C-relation
       by \newline$C(x,y,z)\Leftrightarrow v(y\iz)>v(x\iz)$. Such a C-relation is called a \textbf{valuational} C-relation.
             \end{enumerate}           
   If $(G,C)$ is a C-group, then we say that $C$ is a C-relation \textbf{ of elementary type} if it is either order-type or valuational.   
          \end{Ex}

We say that a structure $\mathcal{M}=(M,C,\dots)$ endowed with a C-relation is \textbf{C-minimal} if for every 
  $\mathcal{N}=(N,C,\dots)$ such that $\mathcal{N}\equiv\mathcal{M}$ every definable subset of $N$ is
  quantifier-free definable in the language $\{C\}$.
  If $\T$ is a meet-semilattice tree and $M$ a set of maximal branches of $\T$, then 
we can define a C-relation on $M$ as follows:	$C(x,y,z)$ holds if and only if the branching point of 
$x$ and $z$ lies strictly 
below the branching point of $y$ and $z$. Conversely, if $(M,C)$ is an arbitrary C-structure, then	
   we can canonically associate a meet-semilattice tree $\T$, called the canonical tree of $(M,C)$, so that 
  $(M,C)$ is isomorphic to a set of maximal branches of $\T$ endowed with the C-relation given above.
  To study C-minimal structures it might be practical to consider their canonical tree.
 In \cite{Macstein}, the authors described dense C-minimal groups by looking at the action induced by the group on its 
 canonical tree. We will do the same in Section \ref{Cminsection}.

  A \textbf{quasi-order} (q.o) is a binary relation which is reflexive and transitive. If 
     $\precsim$ is a quasi-order on a set $A$, then
 it induces an equivalence relation on $A$ by $a\sim b$ if and only if $a\precsim b\precsim a$. 

 \begin{Notation}
  The symbol $\precsim$ will always denote a quasi-order, whereas $\leq$ will always denote an order.
 The symbol $\sim$ will always denote the equivalence relation induced by the  quasi-order $\precsim$ and
  $cl(a)$ will denote the class of $a$ for this equivalence relation.
 The notation $a\precnsim b$ means $a\precsim b\wedge a\nsim b$. If $S,T$ are two subsets of a quasi-ordered set $(A,\precsim)$, then
  the notation $S\precsim T$ (respectively $S\precnsim T$) means that $s\precsim t$ (respectively 
  $s\precnsim t$) for any $(s,t)\in S\times T$. If $a\in A$, then we write $S\precsim a$ instead of $S\precsim\{a\}$.
 \end{Notation}
 The q.o $\precsim$ 
 induces an order on the quotient $A/\sim$ by $cl(a)\leq cl(b)$ if and only if $a\precsim b$.
 We say that a q.o $\precsim$ is \textbf{total} 
 if for every $a,b\in A$, either $a\precsim b$ or $b\precsim a$ holds. 
 Note that $\precsim$ is total if and only if it induces a total order on
 $A/\sim$. \textit{Unless explicitly stated otherwise, every q.o considered in this paper is total}.
 
 A \textbf{coarsening} of $\precsim$ is a q.o $\precsim^{\ast}$ such that
      $a\precsim b\Rightarrow a\precsim^{\ast}b$ for any $a,b\in A$. In that case, we also say that $\precsim$ is a
      \textbf{refinement} of $\precsim^{\ast}$.
      The \textbf{trivial q.o} on $A$ is the q.o which only has one equivalence class, i.e $a\precsim b$ for every 
      $a,b\in A$. We usually denote it by $\precsim_t$.
 If $a,c,b\in A$, then we say that $c$ is \textbf{between a and b} if $a\precsim c\precsim b$ or $b\precsim c\precsim a$ holds. If the stronger condition 
 $a\precnsim c\precnsim b\vee b\precnsim c\precnsim a$ holds, then we then say that $c$ is \textbf{strictly between} $a$ and $b$.
 If $S$ is a subset of $A$, then we define the \textbf{maximum} (respectively \textbf{minimum}) of $S$ as the set of all 
 elements $s$ of $S$ such that $t\precsim s$ (respectively $s\precsim t$) for every $t\in S$. We denote it by 
 $\max(S)$ (respectively $\min(S)$). Note that the 
 maximum of $S$ is always defined but can be empty.
We say that $S$ is:
  \begin{itemize}
   \item  An \textbf{initial segment} if  $s\in S$ and 
  $a\precsim s$ implies $a\in S$.
   \item \textbf{Convex} if $s,t\in S$ and $s\precsim a\precsim t$ implies $a\in S$.
   \item \textbf{Strictly convex} if $s,t\in S$ and $s\precnsim a\precnsim t$ implies $a\in S$.
   \item \textbf{Left-convex} (respectively, \textbf{right-convex}) if 
   $s,t\in S$ and $s\precsim a\precnsim t$  (respectively $s\precnsim a\precsim t$) implies $a\in S$.
  \end{itemize}
  If $S$ is strictly convex, then we define the \textbf{convexity complement} of $S$ as the smallest subset $T$ 
  of $A\backslash S$ such 
  that $S\cup T$ is convex. Note that being left-convex or right-convex implies being strictly convex.
 We can characterize strict convexity by the following lemma:
 \begin{Lem}\label{strictconvexity}
  For any $S\subseteq A$, $S$ is strictly convex if and only if one of the following conditions holds:
  \begin{enumerate}[(i)]
   \item $S$ is convex. In that case the convexity complement of $S$ is $\varnothing$.
   \item $\min(S)\neq\varnothing$  and $S\cup cl(m)$ is convex for any $m\in\min(S)$.
   In that case $S$ is right-convex and its convexity complement is
   $cl(m)\backslash S$.
   \item $\max(S)\neq\varnothing$  and $S\cup cl(M)$ is convex for any $M\in\max(S)$. In that case $S$ is left-convex and its convexity complement is
   $cl(M)\backslash S$.
   \item $\min(S),\max(S)$ are both non-empty and $S\cup cl(m)\cup cl(M)$ is convex for any 
   $m\in\min(S)$ and $M\in\max(S)$. In that case 
   the convexity complement of $S$ is
   $(cl(m)\cup cl(M))\backslash S$.
  \end{enumerate}

 \end{Lem}

 \begin{proof}
  It is easy to check that if one of these conditions holds, then $S$ is strictly convex. Let us prove the converse. Assume
  that $S$ is not convex. This means that there exists $m,t\in S$ and $a\notin S$ such that 
  $m\precsim a\precsim t$. However, since $S$ is strictly convex, we cannot have $m\precnsim a\precnsim t$. Without loss of 
  generality, we can thus assume that $m\sim a$. Assume that $m\notin \min(S)$ and $m\notin\max(S)$. Then there are 
  $s,M\in S$ with $s\precnsim a\sim m\precnsim M$. Since $S$ is strictly convex, it follows that $a\in S$, which is a contradiction.
  Thus, we either have $m\in\min(S)$ or $m\in\max(S)$. 
  If $S\cup cl(m)$ is convex, then we are in case (ii) or (iii). Assume then that it is not convex. 
  Without loss of generality, we may assume $m\in\min(S)$.
  Take $b\notin S\cup cl(m)$ and 
  $M\in S\cup cl(m)$ with $m\precnsim b\precsim M$. Since $M\notin cl(m)$, we have $M\in S$.  
  By strict convexity of $S$, we must have $b\sim M$. If $M\notin\max(S)$, then we would have $m\precnsim b\precnsim M'$ for a
  certain $M'\in S$, which would imply $b\in S$. Therefore, we must have $M\in\max(S)$. Now let us proves that $S\cup cl(m)\cup cl(M)$
  is convex, so that we are in case (iv). Let $c\in A$ such that there is $s,t\in S\cup cl(m)\cup cl(M)$ with 
  $s\precsim c\precsim t$. Since $m,M$ are respectively minimal and maximal in $S$, we have 
  $m\precsim c\precsim M$. If $c\notin cl(m)\cup cl(M)$, then we even have $m\precnsim c\precnsim M$. By strict convexity
  of $S$, this implies $c\in S$.
  The statements about the convexity complement are clear.
  \end{proof}

     In this paper, a \textbf{quasi-ordered group} is just a group endowed with a quasi-order without any further
 assumption. An element $g$ of a quasi-ordered group $(G,\precsim)$ is called \textbf{v-type} if $g\sim\ig$ and \textbf{o-type} if $g=1\vee g\nsim\ig$. Moreover, $g$ is 
called \textbf{o$^+$-type} if $\ig\precnsim g$ and \textbf{o$^-$-type} if $g\precnsim\ig$.  Note that $1$ is the only element which is both v-type and o-type.
 If $(G,v)$ is a valued group, then $v$ induces a quasi-order on $G$ via 
 $g\precsim h\Leftrightarrow v(g)\geq v(h)$.
 If $(G,\precsim_G)$ and $(H,\precsim_H)$ are two quasi-ordered groups and $\phi: G\to H$ a group homomorphism, then
 we say that $\phi$ is \textbf{quasi-order-preserving} if for any $g,h\in G$, 
 $g\precsim h$ if and only if $\phi(g)\precsim\phi(h)$.
 It will be convenient to consider quotients, which is why we need the following lemma from \cite{Lehericy}:
 
      \begin{Lem}\label{quotient}
 Let $(G,\precsim)$ be a quasi-ordered group
 and $H$ a normal subgroup of  $G$ such that the following condition is satisfied:
  \[\forall g_1,g_2\in G((g_1\ig_2\notin H\wedge g_1\precsim g_2)\Rightarrow (\forall h_1,h_2\in H, g_1h_1\precsim g_2h_2\wedge h_1g_1\precsim h_2g_2)).\]
 Then $\precsim$ induces a  q.o on the quotient group $G/H$ defined by:
 \[gH\precsim hH\Leftrightarrow g\ih\in H\vee (g\ih\notin H\wedge g\precsim h).\]
\end{Lem} 
 
 Lemma \ref{quotient} was only proved for abelian groups in \cite{Lehericy}, but we can easily see that the proof is exactly 
 the same in the general case.
 The opposite process of quotienting a q.o is lifting, which we will also need.
 Let $G$ be an abelian group and $v:G\to \Gamma\cup\{\infty\}$ a valuation.
 Assume that for each $\gamma\in\Gamma$, the quotient $G^{\gamma}/G_{\gamma}$ is endowed with a q.o $\precsim_{\gamma}$.
 We define the \textbf{lifting} of $(\precsim_{\gamma})_{\gamma\in\Gamma}$ to $G$ as the quasi-order defined on 
 $G$ by the following formula:
 \[g\precsim h\Leftrightarrow v(g)>v(h)\vee (v(g)=v(h)=\gamma\wedge gG_{\gamma}\precsim_{\gamma} hG_{\gamma}).\]

  Let us check that $\precsim$ is indeed a q.o. Reflexivity is clear. Assume  $f\precsim g\precsim h$. If $v(f)>v(g)$ or $v(g)>v(h)$, then clearly $v(f)>v(h)$, hence $f\precsim h$. 
 Thus, we can assume $v(f)=v(h)=v(g)=\gamma$. But then $f\precsim h$ follows from the transitivity of $\precsim_{\gamma}$. Assume now that 
$g\precsim h$ does not hold. In particular, we must have $v(g)\leq v(h)$. If $v(g)<v(h)$, then $h\precsim g$. If 
$v(h)=v(g)=\gamma$, then we cannot have $gG_{\gamma}\precsim_{\gamma} hG_{\gamma}$, but since $\precsim_{\gamma}$ is total it follows that $hG_{\gamma}\precsim_{\gamma} gG_{\gamma}$, 
hence $h\precsim g$. This shows that $\precsim$ is total.

 \section{C-quasi-orders}\label{Cqosection}
 
 \subsection{Definition and axiomatization}
  
  As mentioned in the introduction, we want to associate a quasi-order to every compatible C-relation. This idea originates from the following
  general fact:
  
  \begin{Lem}
   Let $A$ be a set (not necessarily a group), $C$ a C-relation on $A$ and take $z\in A$.
   Then $z$ induces a  quasi-order on $A$ by
   $a\precsim b\Leftrightarrow\neg C(a,b,z)$.
  \end{Lem}
\begin{proof}
 Note that $\neg C(z,z,z)$ follows from $(C_2)$, so we have $z\precsim z$. Let $a\in A$ with $a\neq z$. 
 By $(C_4)$, we have $C(z,a,a)$. By $(C_2)$, this implies $\neg C(a,z,a)$, which by $(C_1)$ implies
 $\neg C(a,a,z)$. This proves that $\precsim$ is reflexive. Transitivity is the contra-position of axiom $(C_3)$. Totality is
 given by axiom $(C_2)$.
 
\end{proof}

In the context of groups, the natural candidate for the parameter $z$ is $z=1$, hence the following definition:

\begin{Def}\label{deffondamentale}
 Let $G$ be a group. For any compatible C-relation
  $C$  on $G$, we define  the \textbf{q.o induced by $C$} as the q.o given by the formula
 $x\precsim y\Leftrightarrow \neg C(x,y,1)$. 
 A \textbf{C-quasi-order} (C-q.o) on $G$ is the q.o induced by a 
 compatible C-relation on $G$. A \textbf{C-quasi-ordered group} (C-q.o.g) is a pair $(G,\precsim)$ consisting of a
 group $G$ endowed with a C-q.o $\precsim$.
\end{Def}
\begin{Rem}
 If $\precsim$ is the q.o induced by $C$, then we have 
 $C(x,y,1)\Leftrightarrow y\precnsim x$.
\end{Rem}

If $\precsim$ is a C-q.o induced by the C-relation $C$, then
we say that $\precsim$ is \textbf{order-type} (respectively \textbf{valuational}/ \textbf{of elementary type}) if $C$ is order-type 
 (respectively valuational/ of elementary type ). These definitions make sense thanks to the following proposition:
 
\begin{Prop}\label{corrbij}
 Let $\precsim$ be a C-q.o. Then there is only one compatible C-relation inducing it, namely the one given by the 
 formula $C(x,y,z)\Leftrightarrow y\iz\precnsim x\iz$.
\end{Prop}

 \begin{proof}
  Let $C$ be a compatible C-relation inducing $\precsim$. $C$ is compatible so we have 
  $C(x,y,z)\Leftrightarrow C(x\iz,y\iz,1)\Leftrightarrow y\iz\precnsim x\iz$.
 \end{proof}
 
 We now want to axiomatize the class of C-q.o's.  
Proposition \ref{corrbij} states that
  $\precsim$ is a C-q.o if and only if the formula
 $y\iz\precnsim x\iz$ defines a compatible C-relation. We thus want to answer the question:  
  When does this formula define a compatible C-relation?
  
  \begin{Lem}\label{C2C3}
    Let $\precsim$ be a quasi-order on a group $G$ and define a ternary relation
    $C(x,y,z)$ by the formula $y\iz\precnsim x\iz$. Then the relation $C$ satisfies $(C_2)$ and $(C_3)$.
  \end{Lem}

  \begin{proof}
     $C$ clearly satisfies $(C_2)$.  
     Assume $C(x,y,z)$ and 
   $\neg C(w,y,z)$ hold. This means $y\iz\precnsim x\iz$ and
    $\neg (y\iz\precnsim w\iz)$. Since $\precsim$ is total, this 
    implies $w\iz\precsim y\iz\precnsim x\iz$, hence $w\iz\precnsim x\iz$ i.e $C(x,w,z)$. 
    This proves
    $(C_3)$.
  \end{proof}

  This gives us an axiomatization of C-q.o's:

\begin{Prop}[Axiomatization of C-q.o's]
 Let $G$ be a group and $\precsim$ a q.o on $G$. Then
 $\precsim$ is a C-q.o if and only if the following three axioms are satisfied:
 \begin{itemize}
   \item[$(CQ_1)$] $\forall x\in G\backslash\{1\}$, $1\precnsim x$.
   \item[$(CQ_2)$] $\forall x,y (x\precsim y\Leftrightarrow x\iy\precsim \iy)$.
   \item[$(CQ_3)$] $\forall x,y,z\in G$, $x\precsim y\Leftrightarrow x^z\precsim y^z$.
  \end{itemize}
\end{Prop}

Note that ``$\Leftrightarrow$'' can be replaced by ``$\Rightarrow$'' in $(CQ_2)$ and $(CQ_3)$ since $(x\iy)(\iy)^{-1}=x$, $(\iy)^{-1}=y$ and $(x^z)^{\iz}=x$.

\begin{proof}
 Define $C(x,y,z):=y\iz\precnsim x\iz$. By Proposition \ref{corrbij}, $\precsim$ is a C-q.o if and only if 
 $C$ is a compatible C-relation.
 Assume  $C$ is a compatible C-relation.
 By $(C_4)$, we have 
 $C(x,1,1)$ for any $x\neq 1$, which means $1\precnsim x$. Take $x,y,z\in G$ with 
 $x\precsim y$, which means $\neg C(x,y,1)$. By $(C_1)$, we then have $\neg C(x,1,y)$. By compatibility, this implies
 $\neg C(x\iy,\iy,1)$ i.e $x\iy\precsim\iy$, hence $(CQ_2)$. By compatibility we also have 
 $\neg C(x^z,y^z,1)$, hence $(CQ_3)$. 
 Conversely, assume $(CQ_1),(CQ_2),(CQ_3)$ hold. By Lemma \ref{C2C3}, we already know that 
 $C$ satisfies $(C_2)$ and $(C_3)$. We first prove that $C$ is compatible.
 Take $x,y,z,u,v\in G$ with $C(x,y,z)$. We thus have $y\iz\precnsim x\iz$. By $(CQ_3)$, this implies
 $uy\iz\iu\precnsim ux\iz\iu$ i.e 
 $(uyv)(uzv)^{-1}\precnsim (uxv)(uzv)^{-1}$, so $C(uxv,uyv,uzv)$. This proves compatibility. Let $x\neq y$ in $G$. 
 $(CQ_1)$ implies $1\precnsim x\iy$ which means $C(x,y,y)$, so $C$ satisfies $(C_4)$. Now assume $\neg C(x,y,z)$, i.e 
 $x\iz\precsim y\iz$. By applying $(CQ_2)$ to this inequality, we get
 $x\iy\precsim z\iy$, hence $\neg C(x,z,y)$, which proves that $C$ satisfies $(C_1)$.
\end{proof}

\begin{Rem}
  By combining $(CQ_3)$ and $(CQ_2)$ we obtain an improved version of $(CQ_2)$:  \newline
  $x\precsim y\Rightarrow x\iy\precsim\iy\wedge \iy x\precsim\iy$.  
  We will also often use the contra-position of $(CQ_2)$:
  \[(CQ_2')\quad y\precnsim x\Rightarrow \iy\precnsim x\iy.\]
\end{Rem}

 \subsection{C-q.o's of elementary type}

  Before investigating the structure of an arbitrary C-q.o.g, we want to understand the structure of C-q.o's of elementary type.
  Assume first that $\precsim$ is a valuational C-q.o on $G$. We then have 
  $g\precsim h\Leftrightarrow v(g)\geq v(h)$ for any $g,h\in G$. In other words, $\precsim$ is the C-q.o. induced by $v$.
  The order-type case is a bit more complicated. Note first that if we start with an ordered group $(G,\leq)$, if $C$ 
  is the C-relation induced by $\leq$ and if $\precsim$ is the corresponding C-q.o, then there is no reason for $\leq$ and 
  $\precsim$ to be the same. In fact, an order-type C-q.o can never be an order.
  Let us have a closer look at $\precsim$.
  It is easy to see 
  from the definition of $C$ and $\precsim$ that $x\precsim y$ is equivalent to the formula
  $(x=y\vee y\neq1)\wedge(x\leq y\vee x\leq1)$. From this formula we immediately see that the following holds:
  \begin{enumerate}[(i)]
   \item If $\{x,y\}<1$, then $x\sim y$.
   \item If $x<1<y$, then $x\precnsim y$.
   \item If $1<x$ and $1<y$, then $x\precsim y\Leftrightarrow x\leq y$.
  \end{enumerate}

  In other words, $\precsim$ is given by:  
  $1\precnsim (G^-,\precsim_t)\precnsim (G^+,\leq)$, where $\precsim_t$ is the trivial q.o on $G^-$.
  This structure completely characterizes order-type C-q.o's:
  
  \begin{Prop}\label{ordertypegroup}
     Let $(G,\precsim)$ be a C-q.o.g. The C-q.o
     $\precsim$ is order-type if and only if
    there exists a subset $G^+$ of $G$ such that the following holds:
     \begin{enumerate}[(i)]
      \item $G=G^+\sqcup G^-\sqcup \{1\}$ (disjoint unions), where $G^-:=\{\ig\mid g\in G^+\}$.
      \item $1\precnsim G^-\precnsim G^+$.
      \item $\precsim$ is trivial on $G^-$ and coincides with an order $\leq$ on $G^+$.
     \end{enumerate} 
    \end{Prop}

    \begin{proof}
    We already showed that order-type q.o's satisfy this condition with $G^+=\{g\mid 1<g\}$. Let us prove the converse.
    We denote by $C$ the C-relation corresponding to $\precsim$.    
  Define an order on $G^-$ as follows:
 $h\leq g\Leftrightarrow \ih\geq \ig$. Now extend $\leq$ on all of $G$ by 
 $G^-<1<G^+$. Note that $x\in G^-$ if and only if $x\precnsim \ix$. By $(CQ_3)$, it then follows that $G^-$ and $G^+$ are stable under 
 conjugation. We first want to show that $(G,\leq)$ is an ordered group. This will be a
 consequence of the following claim:\\

 \underline{Claim:} For any $x,y\in G, x\iy<1\Leftrightarrow x<y$.
 
 \underline{Proof:} Note that it is sufficient to prove $\Rightarrow$. Indeed, assume $\Rightarrow$ has been proved,
 and assume $\neg(x\iy<1)$. This implies $y\ix\leq1$, which by $\Rightarrow$ implies $y<x\vee y=x$, so $\neg(x<y)$.
 
 Assume then that $x\iy<1$.
 The case $x=1\vee y=1$ is clear, so assume $y\neq1\wedge x\neq1$. Since $x\neq y$, it is sufficient to prove
 $x\leq y$.
 If $y\precnsim x$, then by $(CQ_2')$ we have $\iy\precnsim x\iy$. Since $\precsim$ is trivial on $G^-$, this
 implies $x\iy\in G^+$, which contradicts $x\iy<1$. 
 Thus, we have $x\precsim y$. We consider two cases:
 
  Case 1: $y\in G^+$. It follows immediately from $x\precsim y$ that $x\leq y$. 
  
  Case 2: Assume $y\in G^-$. It follows from 
 $1\neq x\precsim y$ that $x\in G^-$. Note that 
 we have $y\ix\in G^+$.  By conjugation, this implies $\ix y\in G^+$, hence $y\precnsim \ix y$.
 By $(CQ_2')$, this implies $\iy\precnsim \ix$, hence $x\leq y$.
  This proves the claim.\\
 
 Now let us show that $(G,\leq)$ is an ordered group. Assume $x<y$ and take $z\in G$.
 By the claim, we have $x\iy<1$, hence $xz(yz)^{-1}<1$, hence $xz<yz$. By conjugation, we also have 
 $\iy x<1$, so $(zy)^{-1}(zx)<1$. By the claim, this means $zx<zy$. This proves that $(G,\leq)$ is an ordered group.

 Denote by $C'$ the C-relation induced by $\leq$. We show that $C'=C$.
 Assume $C(x,y,z)$ holds. The case $x\neq y=z$ is obvious, so assume
 $y\neq z$. 
 We have $y\iz\precnsim x\iz$ and  $z\iy\precnsim x\iy$. We either have 
 $z\iy\in G^+$ or $y\iz\in G^+$. Without loss of generality, we can assume that the former holds (the other case is done similarly).
 We then have $z\iy, x\iy\in G^+$ with
 $z\iy\precnsim x\iy$, which means $1<z\iy<x\iy$. It follows that 
 $y<z<x$, hence $C'(x,y,z)$. 
 Conversely,
 assume $C'(x,y,z)$ holds. Since $C'$ is compatible, this implies 
 $C'(x\iz,y\iz,1)$, which means $1<x\iz$ and $y\iz<x\iz$. We thus have $x\iz\in G^+$ and $y\iz<x\iz$, which means
 $y\iz\precnsim x\iz$, hence $C(x,y,z)$
\end{proof}
  
  All of this shows us how to
   construct $\precsim$ from $\leq$ and vice-versa. More precisely, we see that $\leq$ and $\precsim$ define the same sets:
   
   \begin{Prop}
    Let $(G,\leq)$ be an ordered group and $\precsim$ the corresponding C-q.o. The relation 
    $\precsim$ is quantifier-free definable in the language $\{1,.,^{-1},\leq\}$  and $\leq$ is 
    quantifier-free definable in $\{1,.,^{-1},\precsim\}$.
   \end{Prop}

   \begin{proof}
    As already mentioned, $x\precsim y$ is equivalent to $(x=y\vee y\neq1)\wedge(x\leq y\vee x\leq1)$.
    Conversely, $x\leq y$ is equivalent to the formula:\newline
    $(x,y\in G^+\wedge x\precsim y)\vee(x,y\in G^-\wedge \iy\precsim \ix)\vee(x\in G^-\wedge y\in G^+\cup\{1\})
  \vee(x=1\wedge y\in G^+)$, and $G^+$ and $G^-$ are respectively defined by the formulas 
  $\ix\precnsim x$ and $x\precnsim \ix$.
   \end{proof}
  
\begin{Rem}\label{lookslikeexplain}
   We just saw what C-q.o groups of elementary type look like. In Section \ref{structuresection}, our work will consist in showing 
that any C-q.o group is in some sense a ``mix'' of the elementary ones. This means that we will identify parts of the group where the q.o 
is  ``order-type-like'' and parts where it is ``valuational-like''. Intuitively, we want to say that a q.o is ``like'' an elementary-type q.o on a subset $T$ of $G$ if 
it shares the important properties of this elementary q.o.
We will say that the q.o $\precsim$ is \textbf{valuational-like  on $T$} if 
$gh\precsim\max(\{g,h\})$ for any $g,h\in T$. We will say that $\precsim$ is \textbf{order-type-like  on $T$} if $T$ can be partitioned into two subsets, 
$T^-$ and $T^+$, such that the following holds: $T^-=\{\ig\mid g\in T^+\}$, $T^-\precnsim T^+$ and $\precsim$ is trivial on $T^-$ (i.e 
$g\sim h$ for all $g,h\in T^-$). We say that $\precsim$ is \textbf{elementary-type-like on $T$} if it is either valuational-like or order-type-like on $T$.
  \end{Rem}

  \subsection{Connection with compatible q.o's}\label{compatiblesection}
  
  We now want to establish the connection between the notion of C-q.o developed in this paper and the work done
  in \cite{Lehericy} which we mentioned in the introduction. 
  As we showed in \cite{Lehericy}, we can associate a compatible C-relation to any compatible quasi-order defined 
  on an abelian group. However, this does not mean that compatible q.o's are C-q.o's. In fact, we have the following:
  
  \begin{Prop}\label{compatibleqo}
   Let $(G,\precsim)$ be a compatible quasi-ordered abelian group. Then $\precsim$ is a C-q.o if and only if
  every element of $G$ is v-type. 
  \end{Prop}

  \begin{proof}
   By Proposition 2.13 of \cite{Lehericy}, we know that the set $G^o$ of all o-type elements of $G$ is a 
   subgroup of $G$ and that $(G^o,\precsim)$ is an ordered abelian group. If $G^o$ is non-trivial, then 
   $G$ contains negative elements, which contradicts axiom $(CQ_1)$, so $\precsim$ cannot be a C-q.o.
   Thus, $G^o$ must be trivial, which means that every element of $G$ is v-type.
  \end{proof}

  Now let $(G,\precsim)$ be a compatible quasi-ordered abelian group.
  Proposition \ref{compatibleqo} states that, if the subgroup $G^o$ of o-type elements is non-trivial, then 
  $\precsim$ is not a C-q.o.
  However, we can transform $\precsim$ into a corresponding 
  C-q.o $\precsim^{\ast}$.
  We know that $\precsim$ coincides with an order $\leq$ on $G^o$ and is valuational on $G^v$. Now define $\precsim^{\ast}$
  as follows: on $G^o$, $\precsim^{\ast}$ is the order-type C-q.o corresponding to $\leq$. On 
  $G^v$, $\precsim^{\ast}$ coincides with $\precsim$. Finally, declare $G^o\precnsim^{\ast} G^v$. Then 
  $\precsim^{\ast}$ is a C-q.o. Now denote by $C^{\ast}$ the C-relation corresponding to the C-q.o $\precsim^{\ast}$ and denote by $C$ the C-relation 
  induced by the compatible q.o $\precsim$ as defined in Proposition 4.1 of \cite{Lehericy}. 
  We recall that in \cite{Lehericy}, we defined the C-relation induced by $\precsim$ as a sort of ``mix'' between the definition of 
  a C-relation induced by an order and the C-relation induced by a valuation. More precisely, Proposition 4.1 of 
  \cite{Lehericy} defines $C$ as follows: the relation $C(x,y,z)$ holds if and only if the following formula is true:\newline  
  $(x\neq y=z)\vee(x\iz\in G^v\wedge (y\iz\precnsim x\iz))
    \vee(y\iz,x\iz\in G^o
   \wedge (1\precnsim x\iy\wedge 1\precnsim x\iz))$.  
  By distinguishing the cases $x\iz\in G^v$ and $x\iz\notin G^v$, one can 
show that $C(x,y,z)$ holds if and only if $y\iz\precnsim^{\ast} x\iz$. It then follows that $C=C^{\ast}$.


\section{Structure of C-q.o.g's}\label{structuresection}
 
In this section we describe the structure of an arbitrary C-q.o.g $(G,\precsim)$. We start by giving four different examples of 
C-q.o's. All of them are obtained by lifting (with the notion of lifting defined after Lemma \ref{quotient}). 
 It is possible to directly check that each of them satisfy the axioms 
 of C-q.o's, but this will actually be a consequence of Propositions \ref{liftingweak}, \ref{semidirectproductlift} and \ref{weldingconstruction}. 
 Examples (a),(c) and (d) are obtained by direct application of \ref{liftingweak}, and example (e) is proved from example (d) 
 with Proposition \ref{semidirectproductlift}. Finally, to prove example (b), 
 apply Proposition \ref{weldingconstruction} on the C-q.o group $(G,\precsim)$ 
 from example (a) with $g:=(-1,0)$.

 \begin{Exs}\label{examplesfondamentales}
 
 Set $G:=\entiersrel^2$. We let $\precsim_o$ denote the C-q.o induced by the usual order of $\Z$ (which is characterized in Proposition \ref{ordertypegroup}) and $\precsim_v$ 
 the C-q.o induced by the trivial 
  valuation on $\Z$. Define the valuation $v_G$ on $G$ by 
 $v_G(a,b)=\left\{\begin{array}{cc}
                1 &\text{ if } a\neq0.\\
                2 &\text{ if } a=0\neq b.\\
                \infty & \text{ if } a=b=0.
               \end{array}\right.$
  
               We have $G^1/G_1\cong G^2/G_2\cong\entiersrel$.
  We define three different C-q.o's on $G$:
  \begin{enumerate}[(a)]
   \item Choose $\precsim_1:=\precsim_o$ and $\precsim_2:=\precsim_v$. 
   The lifting of $(\precsim_1,\precsim_2)$ to $G$ is the C-q.o given by :
   
   $(0,0)\precnsim (\{0\}\times(\Z\backslash\{0\}),\precsim_t)\precnsim (-\entiers\times\entiersrel,\precsim_t)
 \precnsim (\entiers\times \entiersrel,\precsim)$,
 
   where $\precsim_t$ always denotes the trivial q.o and $\precsim$ is defined on $\entiers\times \entiersrel$ as follows: 
 $(a,b)\precsim (c,d)\Leftrightarrow a\leq c$. In this example, $\precsim$ is valuational on $\{0\}\times\Z$ and 
  order-type-like  on $(\Z\backslash\{0\}\times\Z)$. The set of v-type elements is $\{0\}\times\Z$, the set of 
$o^-$-type elements is $-\N\times\Z$ and the set of $o^+$-type elements is $\N\times\Z$.
 \item Coarsen the C-q.o of the previous example by declaring that \newline
 $(\{0\}\times(\Z\backslash\{0\}),\precsim_t)\sim(-\entiers\times\entiersrel,\precsim_t)$. This new C-q.o is now given by:
 
 $(0,0)\precnsim ((-\N_0\times\Z)\backslash\{(0,0)\},\precsim_t)
 \precnsim (\entiers\times \entiersrel,\precsim)$.
  
   All elements of $G$ in this example have the same type as in (a).
 \item Define $\precsim_1=\precsim_2=\precsim_o$. The lifting of  
 $(\precsim_1,\precsim_2)$ to $G$ is the C-q.o:
 
 $(0,0)\precnsim (\{0\}\times-\entiers,\precsim_t)\precnsim (\{0\}\times\entiers,\leq)\precnsim
 (-\entiers\times\entiersrel,\precsim_t)\precnsim(\entiers\times\entiersrel,\precsim)$,
 
 where $\leq$ is the natural order of $\entiersrel$ and $\precsim$ is defined on $\entiers\times \entiersrel$ as follows: 
 $(a,b)\precsim (c,d)\Leftrightarrow a\leq c$. Here $\precsim$ is order-type-like  on 
 $\{0\}\times\Z$ and on $(\Z\backslash\{0\}\times\Z)$. The set of $o^-$-type elements is 
 $\{0\}\times-\N\cup-\N\times\Z$, the set of $o^+$-type elements is 
  $\{0\}\times\N\cup\N\times\Z$, and $(0,0)$ is the only v-type element. \\

 \item Let $\precsim$ be the C-q.o of example (a) on $G$. 

Set  
 $H:=\coprod_{\Z}G=\{(g_n)_{n\in\Z}\in G^{\Z}\mid\text{ the support of $(g_n)_{n\in\Z}$ is finite}\}$ ($H$ is thus the Hahn sum of $\Z$-many copies of $G$). We denote the elements of $H$ as 
 formal sums $h=\sum_{n\in\Z}g_n\tau_n$.
 $H$ can be endowed 
 with a valuation $w_H:H\to \Z\cup\{\infty\}$, where $w_H(h)$ is defined as the minimum of the support of $h$. In this context, we have 
 $H^{\gamma}=\{\sum_{n\in\Z}g_n\tau_n\mid \forall n<\gamma, g_n=(0,0)\}$,
$H_{\gamma}=\{\sum_{n\in\Z}g_n\tau_n\mid \forall n\leq\gamma, g_n=(0,0)\}$ and
 $H^{\gamma}/H_{\gamma}\cong G$ for every $\gamma\in\Z$.
We endow 
  $H$ with the lifting $\precsim_H$ of $(\precsim)_{\gamma\in\Z}$. 
 Here the C-q.o alternates infinitely many times between order-type-like parts and valuational-like parts. More precisely, 
for any $h=\sum_{n\in\Z}g_n\tau_n\in H$ with $\gamma:=w_H(h)$, then $h$ is v-type if and only if $g_{\gamma}\in \{0\}\times\Z$, 
$h$ is $o^-$-type if and only if $g_{\gamma}\in -\N\times\Z$ and $h$ is $o^+$-type if and only if
$g_{\gamma}\in \N\times\Z$. For any $\gamma\in\Z$, $\precsim_H$ is valuational-like on 
$\{h=\sum_{n\in\Z}g_n\tau_n\in H\mid w_H(h)=\gamma, g_{\gamma}\in \{0\}\times\Z\}$ and 
is order-type-like on $\{h=\sum_{n\in\Z}g_n\tau_n\in H\mid w_H(h)=\gamma,g_{\gamma}\in(\Z\backslash\{0\})\times\Z\}$.
\\
 
 We can also give a non-abelian example:

  \item Let $(H,\precsim_H)$ be  as in the previous example. For any $k\in \Z$, let $\alpha_k$ be the 
 $k$-th shift on $H$ (i.e $\alpha_k(\sum_{n\in\Z}g_n\tau_n)=\sum_{n\in\Z}g_{n-k}\tau_n$). This is a group automorphism of $H$. 
 Set $F:=\Z\ltimes_{\alpha}H$ ($\ltimes_{\alpha}$ denotes the semi-direct product) and define $\precsim_F$ by:\newline
 $(k,h_1)\precsim_F (l,h_2)\Leftrightarrow (k\precsim_v l)\wedge (l\neq 0\vee(l=0\wedge h_1\precsim_H h_2))$.
 Here the elements of $H$ have the same type as in (d).
Elements of the form $(l,h)$ with $l\neq0$ are v-type.
  \end{enumerate}

 \end{Exs}
 
 We see on each of these examples that $G$
 can be partitioned into strictly convex subsets on each of which $\precsim$ is  elementary-type-like. We want
 to show that this is true for an arbitrary C-q.o.g. 
As the terminology and Examples \ref{examplesfondamentales} suggest, it will turn out that $\precsim$ is valuational-like on the set of v-type elements
and  order-type-like  around o-type elements. 
 Note that
 Example (b) seems counter-intuitive. Indeed, we would expect the C-q.o to separate o-type elements from v-type elements,
 but we see that 
$(0,1)\sim (-1,1)$. This means that the C-q.o does not distinguish between the v-type element $(0,1)$ and the o-type element $(-1,1)$. 
This phenomenon is what we call ``welding''.
 We say that $G$ is \textbf{welded} at $h$, or that $h$ is a \textbf{welding point} of $G$ 
 if  there exists an element $g$ such that $g$ and $h$ are of different type and $g\sim h$. 
 We will see that the existence of welding 
 in certain groups makes things technically slightly more difficult but does not fundamentally change the structure 
 of a C-q.o.g.\\

 The following propositions show the relevance of distinguishing o-type elements from v-type elements and 
 justify our terminology:
 
 \begin{Prop}\label{vtypegroup}
  The C-q.o $\precsim$ is valuational if and only if every element of $G$ is v-type.
 \end{Prop}

 \begin{proof}
  If $\precsim$ is valuational, then every element must obviously be v-type. Conversely, assume that
   every element is equivalent to its inverse. We only have to check that the ultrametric inequality is satisfied.
   Let $g,h\in G$. If $h\precsim g$, then by $(CQ_2)$ we have 
  $g\ih\sim h\ig\precsim\ig\sim g$. If $g\precsim h$, then we have $g\ih\precsim\ih\sim h$. In any case, we have
  $g\ih\precsim \max(g,h)$.
 \end{proof}

 \begin{Prop}
   The C-q.o $\precsim$ is order-type if and only if every element of $G$ is o-type and
    $G$ contains exactly one equivalence 
   class of $o^-$-type elements.
 \end{Prop}

 \begin{proof}
  Both directions are proved with \ref{ordertypegroup}.
  If $\precsim$ is order-type, then we see from \ref{ordertypegroup} that every element is o-type and that  all the 
  $o^-$-type elements are contained in one class. For the converse, set
  $G^+:=\{o^+\text{-type elements}\}$ and $G^-:=\{o^-\text{-type elements}\}$. By assumption, $\precsim$ is trivial
  on $G^-$. We obviously 
  have $G=\{1\}\sqcup G^+\sqcup G^-$.
  Let $g\in G^+$. By definition of $o^+$-type, we have $\ig\precnsim g$. By assumption, the elements of $G^-$ are
  all equivalent to one another, hence $G^-\precnsim g$. This shows $1\precnsim G^-\precnsim G^+$. 
  We just have to check that 
  $\precsim$ is an order on $G^+$. Let $g,h\in G^+$ with $g\sim h$. By $(CQ_2)$, 
  $g\precsim h\precsim g$ implies $g\ih\precsim \ih$ and $h\ig\precsim\ig$, so we have 
  $\{g\ih,h\ig\}\precsim G^-$. This is only possible if $g\ih=1$ i.e $g=h$.
 \end{proof}

 \begin{Rem}
 As example \ref{examplesfondamentales}(c) above shows, the fact that every element is o-type is not 
 sufficient to insure that $\precsim$ is order-type.
 \end{Rem}

    \subsection{Some relations between $\precsim$ and the group operation}\label{relationsection}

Here we investigate the relation between multiplication and $\precsim$. 
More precisely, we want to understand how the equivalence
class of the product of two elements relates to the equivalence class of each factor. These results will play 
a fundamental role in the proofs of Section \ref{typecomponents}.
 We fix a C-q.o.g $(G,\precsim)$. We first note that in many cases the order of the factors will not matter:

 \begin{Lem}\label{commuteeqLem}
  For any $g,h\in G$, $hg\sim g\Leftrightarrow gh\sim g$.
 \end{Lem}

 \begin{proof}
  It is a direct consequence of $(CQ_3)$: take the inequalities 
  $hg\precsim g\precsim hg$ and conjugate by $g$.
 \end{proof}

\begin{Lem}\label{lemfond}
 Let $g,h\in G$. The following holds:
 \begin{enumerate}[(i)]
  \item If $h\precnsim \ig$, then $g\sim hg\sim gh$.
  \item Assume that $h\precnsim \{\ig, g\}$. 
 Then $\ih\precnsim \{g,\ig\}$ and we have $gh\sim g\sim g\ih$ and $\ig\sim h\ig\sim \ih\ig$.
 \item If $\{h,\ih\}\precsim \ig\precnsim g$, then $g\sim gh\sim g\ih\sim hg\sim \ih g$ and 

$\ig\sim \ig\ih\sim \ig h\sim \ih\ig\sim h\ig$.
 \end{enumerate}

\end{Lem}
\vspace{5mm}
\begin{proof}
\begin{enumerate}[(i)]
 \item By $(CQ_2)$, $h\precsim \ig\Rightarrow hg\precsim g$. 
 By $(CQ_2')$, $h\precnsim \ig\Rightarrow \ih\precnsim \ig\ih$. 
 By $(CQ_2)$, $\ih\precsim \ig\ih\Rightarrow g\precsim hg$, hence  
  $g\sim hg$.
  \item By  (i), $g\sim gh$ and $\ig\sim \ig h$.  
  By $(CQ_2')$, $h\precnsim gh\Rightarrow \ih\precnsim g$
    and $h\precnsim \ig h\Rightarrow \ih\precnsim \ig$.
 In particular, $\ih$ satisfies
 $\ih\precnsim \{g,\ig\}$, so we get $g\sim g\ih$ and $\ig\sim \ig\ih$, hence the claim.
 \item By (i), $\{h,\ih\}\precnsim g$ implies $\ig\sim \ig\ih\sim \ig h$. 
 By $(CQ_2)$, $h\precsim \ig\Rightarrow hg\precsim g$ and $\ih\precsim \ig\ih\Rightarrow g\precsim hg$, 
 hence $g\sim hg$. Analogously, $g\sim g\ih$. The rest follows from Lemma \ref{commuteeqLem}.
\end{enumerate} 
\end{proof}

We can summarize these results in the following proposition:

\begin{Prop}\label{toutesequivalences}
 Assume $g$ is v-type. If $h\precnsim g$, then $\ih\precnsim g$ and we have 

 $hg\sim\ih g\sim g\ih\sim gh\sim g\sim \ig\sim \ig h\sim \ig\ih\sim\ih\ig\sim h\ig$.
 
 Assume $g$ is $o^+$-type. If $\{h,\ih\}\precsim \ig$, then we have

 $\ig\ih\sim \ig h\sim \ig\precnsim g\sim gh\sim g\ih\sim hg\sim\ih g$.
\end{Prop}

 We now want to find an analog of axiom $(Q_2)$ of compatible q.o's (see \cite{Lehericy}).
\begin{Lem}
 If $f\precsim g$ and $\ig\precsim \ih\ig$, then $fh\precsim gh$ and $hf\precsim hg$.
\end{Lem}

\begin{proof}
 By $(CQ_2)$, $f\precsim g$ implies  $f\ig\precsim \ig$. By assumption, this implies
 $f\ig\precsim \ih\ig$. By $(CQ_2)$ again, this implies $fh\precsim gh$. $(CQ_3)$ then implies $hf\precsim hg$.
\end{proof}

\begin{Prop}\label{addition}
 Let $f,g\in G$ such that $f\precsim g$ and assume that either $g\nsim\ih$ or $\{h,\ih\}\precsim g\precnsim \ig$ holds.
 Then we have $fh\precsim gh$
 and $hf\precsim hg$.
\end{Prop}

\begin{proof}

 If $\ih\precnsim g$, then  by \ref{lemfond} we have $\ig\sim \ih\ig$.
 If $g\precnsim \ih$, then $(CQ_2')$ implies $\ig\precnsim \ih\ig$. In both cases, we have 
 $\ig\precsim \ig\ih$, so we can apply the previous lemma. 
 For the second claim, we use \ref{toutesequivalences} to get $\ig\sim \ih\ig$.
\end{proof}

\begin{Rem}
 We just showed that C-q.o.g's satisfy the formula:
 $\forall g,h,f\in G, f\precsim g\nsim\ih\Rightarrow fh\precsim gh$.
  This formula is very similar to axiom $(Q_2)$ of compatible q.o's and seems to be 
 more practical to deal with than axiom $(CQ_2)$ of C-q.o's. However, we don't know if we can actually replace 
 $(CQ_2)$ by this formula in our axiomatization of C-q.o's.
\end{Rem}

\subsection{Quotients}\label{quotientsection}

 In the theory of ordered abelian groups there is a classical notion of the order induced on a quotient $G/H$ where 
 $H$ is a normal convex subgroup of $G$. In \cite{Lehericy}, we showed that the same holds for compatible quasi-ordered abelian groups.
 Here we show a similar result for C-q.o.g's. However, because of the occasional occurrence of welding, 
     it won't be sufficient for us to only consider convex subgroups, so we will show that a C-q.o $\precsim$ on $G$ canonically induces a C-q.o on the quotient group 
   $G/H$ if $H$ is a normal strictly convex subgroup of $G$.
  This will be useful to describe the structure of the C-q.o on $G$.
 Note first that thanks to axiom $(CQ_1)$ any convex subgroup of $G$ is an initial segment. This also means that any 
 non-convex strictly convex subgroup of $G$ is in case (iii) of Lemma \ref{strictconvexity}. 
 
  \begin{Prop}\label{strictconvexquotientProp}
    Let $H$ be a strictly convex normal subgroup of $G$. Then $\precsim$ induces a C-q.o 
    on $G/H$ by the formula:
    $gH\precsim hH\Leftrightarrow (g\in H)\vee(h\notin H\wedge g\precsim h)$.
    
   \end{Prop}

   The proof of Proposition \ref{strictconvexquotientProp} is done in three parts. We first show the case where $H$ is convex:

 \begin{Prop}\label{quotientconvex}
   Let $(G,\precsim)$ be a C-q.o.g and $H$ a convex normal subgroup of $G$. Then 
  $\precsim$ induces a C-q.o on $G/H$ 
   given by the formula:   
   $gH\precsim hH\Leftrightarrow (g\in H)\vee(h\notin H\wedge g\precsim h)$.   
  \end{Prop}

  \begin{proof}
   We apply Lemma \ref{quotient}. 
   Let $g_1,g_2\in G$ with $g_1\precsim g_2$ and $g_1\ig_2\notin H$ and let $h_1,h_2\in H$. We want to show that 
$g_1h_1\precsim g_2h_2\wedge h_1g_1\precsim h_2g_2$ holds. 
   If $g_1\in H$, then $g_2\notin H$ and we have 
   $h_1g_1, g_1h_1\in H$ and  $h_2g_2, g_2h_2\notin H$. By convexity of $H$, this implies 
   $g_1h_1\precsim g_2h_2\wedge h_1g_1\precsim h_2g_2$. Now assume $g_1\notin H$. By convexity of $H$, this implies $g_2\notin H$. 
 By convexity of $H$, we have $\{h_1,h_2\}\precnsim \{g_1,g_2,\ig_1,\ig_2\}$. By Lemma \ref{lemfond}(iii), this implies
  $h_1g_1\sim g_1h_1\sim g_1\precsim g_2\sim g_2h_2\sim h_2g_2$.
 This proves that the assumption of Lemma \ref{quotient} is verified, so $\precsim$ induces a q.o on $G/H$
 by the formula 
 $gH\precsim hH\Leftrightarrow g\ih\in H\vee (g\ih\notin H\wedge g\precsim h)$. We now want to show that  
 this is equivalent to $ (g\in H)\vee(h\notin H\wedge g\precsim h)$.
 Assume $gH\precsim hH$ and $g\notin H$. 
 If $h\precnsim g$, then by Lemma \ref{lemfond} we have $\ig\sim h\ig$. This  implies $h\ig\notin H$ and 
 $h\precnsim g$, which contradicts the assumption. Thus, $g\precsim h$. Since $g\notin H$, this implies $h\notin H$,
 so $h\notin H\wedge g\precsim h$ holds. 
 Conversely, assume $(g\in H)\vee(h\notin H\wedge g\precsim h)$. 
 If $g\notin H$, then $g\precsim h$, which implies $gH\precsim hH$. If $g\in H$, then either 
 $h\in H$, in which case $g\ih\in H$, or $h\notin H$, which implies $g\ih\notin H\wedge g\precsim h$ by convexity of $H$. 
 In both cases, we have 
 $gH\precsim hH$. 
 Now we can show that the induced q.o is a C-q.o.
 For $(CQ_1)$: If $g\notin H$ and $h\in H$, then by convexity of
 $H$ we have $h\precnsim g$, so $1=hH\precnsim gH$. Now let us prove
 $(CQ_2)\wedge(CQ_3)$. Assume $gH\precsim hH$. If $g\ih\in H$, then by $(CQ_1)$ we have 
 $g\ih H\precsim \ih H$. Since $H$ is normal, we also have 
 $(g\ih)^z\in H$, hence $g^z H\precsim h^z H$.
 If $g\ih\notin H$, then $g\precsim h$, which implies $g\ih\precsim \ih$ and $g^z\precsim h^z$. This implies
 $g\ih H\precsim \ih H$  and $g^z H\precsim h^z H$
  \end{proof}

      If $H$ is only strictly convex, then the assumption of Lemma \ref{quotient} is in general not verified,
     which is why we need the following lemma:
     
\begin{Lem}\label{refinement}
      Let $(G,\precsim_1)$ be a C-q.o.g and 
      let $H$ be a strictly convex normal subgroup of $(G,\precsim_1)$ with convexity complement $F\neq\varnothing$.
      We are then in case (iii) of Lemma \ref{strictconvexity}, so we have $H\precsim F$.
      Let $\precsim_2$ be the refinement of $\precsim_1$
      defined by  declaring that $H\precnsim_2 F$. Then $\precsim_2$ is a C-q.o and
      $H$ is $\precsim_2$-convex.
     \end{Lem}

     \begin{proof}
      The fact that $H$ is $\precsim_2$-convex is clear, as is the fact that 
      $1\precnsim_2 x$ for every $x\in G$.      
      Since $F\neq\varnothing$, $\max(H)$ is non-empty.
      Note that the notation $\max(H)$ is unambiguous, since the max of $H$ in $(G,\precsim_1)$ is the same as in
      $(G,\precsim_2)$.      
      Now assume $x\precsim_2 y$. Since $\precsim_1$ is a coarsening of $\precsim_2$, we have 
      $x\precsim_1 y$. This implies $x\iy\precsim_1 \iy$ and $x^z\precsim_1 y^z$. The only way that we could have 
      $\iy\precnsim_2 x\iy$ is if $\iy\in \max(H)$ and $x\iy\in F$. However, if $\iy\in H$, then $y\in H$. Since we have
      $x\precsim_2 y$, $y\in H$ also implies $x\in H$. This means $x\iy\in H$, so $x\iy\notin F$. It follows that 
      $x\iy\precsim_2 \iy$. By the same reasoning (using the fact that $H$ is normal), we get $x^z\precsim_2 y^z$.
     \end{proof}
     
   We can now show Proposition \ref{strictconvexquotientProp}: 

   \begin{proof}[proof of \ref{strictconvexquotientProp}]
    Set $\precsim_1:=\precsim$ and consider the q.o $\precsim_2$ as in Lemma \ref{refinement}. Since $H$ is
    $\precsim_2$-convex, we know that the formula 
    $gH\precsim hH\Leftrightarrow (g\in H)\vee(h\notin H\wedge g\precsim_2 h)$
    gives a well-defined C-q.o.
    It is easy to see that $(g\in H)\vee(h\notin H\wedge g\precsim_2 h)$ is equivalent to 
    $(g\in H)\vee(h\notin H\wedge g\precsim h) $, since for any $h\notin H$ and any $g\in G$,
    $g\precsim h\Leftrightarrow g\precsim_2 h$.
   \end{proof}

\subsection{Type-components}\label{typecomponents}

In this section, we introduce the ``type-components'' $T_g$ mentioned in the introduction. 
For $g\neq1$, we want to find a set $T_g$ which is the biggest strictly convex subset of $G$ containing $g$ 
on which $\precsim$ is elementary-type-like. 
For an o$^+$-type $g\in G$, one can see that the set of $h\in G$, which are o$^+$-type and such that every element strictly
between $g$ and $h$ are also o$^+$-type is the greatest strictly convex subset of o$^+$-type elements containing $g$. We
can define in the same way such a “strictly convex closure” for an o$^-$-type element or a v-type element. Now,
since by definition $T_g$ contains $g$ and $\ig$, in the o-type cases $T_g$ cannot be this closure. We will show that $T_g$
is the union of the strictly convex closures of $g$ and $\ig$. In the v-type case the strictly convex closures of $g$
and $\ig$ are equal. We also introduce the set $G_g$ which should be thought of as the set of elements of $G$ which
are “below” $T_g$ . We then introduce the set $G^g$ which should be thought of as the set of elements which are not
bigger than $T_g$. We will show that $G_g$ and $G^g$ are subgroups.
For proving the properties of $T_g,G^g,G_g$, and the welding properties, it is more convenient to define $T_g$ by
means of formulas with inequalities instead of strict inequalities. This motivates the following definitions.
For an element $1\neq g\in G$, we define the type-component $T_g$ of $g$ as follows:
\begin{itemize}
 \item If $g$ is v-type, then $T_g$ is the set of v-type elements $h\neq 1$ such that there is no 
 $o^+$-type element between $h$ and $g$.
 \item If $g$ is $o^+$-type, then $T_g^+$ is the set of $o^+$-type elements $h$ such that every element between 
 $g$ and $h$ is $o^+$-type. We then set $T_g^-:=(T_g^+)^{-1}$ and $T_g:=T_g^+\bigcup T_g^-$.
 \item If $g$ is $o^-$-type, then $T_g:=T_{\ig}$.
 
\end{itemize}
We  define two sets $G^g$ and $G_g$ as follows:
  
   \begin{itemize}
    \item If $g$ is v-type, then define $G_g:=\{h\mid h\precnsim T_g\}$.
    \item If $g$ is $o^+$-type, then define $G_g:=\{h\mid \{h,\ih\}\precsim \ig\}$.
    \item If $g$ is $o^-$-type, then define $G_g:=G_{\ig}$.
   \end{itemize}
In all cases we set $G^g:=G_g\bigcup T_g$.
 For $g=1$, we set $T_g=G^g=G_g=\{1\}$. We will show later that $G^g$ and $G_g$ are actually subgroups of $G$ (see Propositions  \ref{otypequotient} and \ref{Ggvtypecase}). Note that for any 
$g\in G$, $1\in G_g$, so $G_g$ and $G^g$ are non-empty.

\begin{Ex}\label{componentsexample}
 Let us have a look again at the groups given in Examples \ref{examplesfondamentales}.
 Set $g=(0,1)$ and $h=(1,0)$. In examples (a), (b) and (c) we have 
 $T_g=(\{0\}\times\entiersrel)\backslash \{(0,0)\}$, $T_h^+=\entiers\times\entiersrel$ and 
 $T_h=(\entiersrel\backslash \{0\})\times\entiersrel$. We also have $G_g=\{0\}, G^g=G_h=\{0\}\times\entiersrel, G^h=G$.
 In other words, we have $G^g=G^2,G_g=G_2,G^h=G^1,G_h=G_1$. It is also easy to see that the q.o induced on the quotients
 $G^g/G_g$ and $G^h/G_h$ are exactly the q.o's $\precsim_1$ and $\precsim_2$ which we lifted to construct the q.o on $G$.
 Note that the only difference between cases (a) and (b)
 is that $T_g,T_h, G_h,G^g$ are convex in case (a) but are only strictly convex in case (b) due to welding.
 Note also that in example (b), each element of the form $(x,y)$ with $x<0$ is an o$^-$-type welding point with $(x,y)\sim (0,z)$ for every 
 $z\neq0$. In particular, 
 there is an o$^-$-type element (for example $(-1,0)$) which is contained between $g$ and $\ig$, even though $\ig\in T_g$. This explains why 
 we restrict to o$^+$-type elements in the definition of $T_g$ when $g$ is v-type.
\end{Ex}

    In the next two sections, we describe some properties of the sets $T_g$, $G^g$ and $G_g$ for $g\neq1$.
 As announced in the introduction, we are going to show  that $T_g$ is a maximal subset of $G$ with the properties that $T_g$
   is strictly convex and that $\precsim$ is  elementary-type-like (of the same type as $g$) on $T_g$ (see Propositions \ref{tgotype} and \ref{tgvtype} ). 
We will also show that $G_g$ and $G^g$ are subgroups of $G$ and that $G_g$ is normal in $G^g$. We 
first show these properties for the case where $g$ is o-type and then do the same for the case where $g$ is v-type. 

 \subsubsection{$T_g$ in the o-type case}
   
   We now want to describe $T_g,G_g,G^g$ in the case where $g\neq 1$ is o-type. By definition of $T_g$, we can  
   assume without loss of generality that $g$ is $o^+$-type. The following proposition states the main properties of $T_g$:

 \begin{Prop}[Characteristics of $T_g$]\label{tgotype}
  The set $T_g$ has the following properties: 
  \begin{enumerate}[(a)]
   \item  $T_g$ is 
  right-convex with convexity complement $F_g:=cl(\ig)\backslash T_g$. Moreover, we have $cl(\ig)=F_g\cup T_g^-$ and $F_g$ can only contain v-type elements.
   \item $T_g$ is convex if and only if $\ig$ is not a welding point of $G$.
   \item   $T_g$ is the biggest strictly convex subset of $G$ containing $g$ with 
  the following properties:
  \begin{enumerate}[(i)]
   \item  Every element of $T_g$ is o-type.
   \item $T_g$ contains exactly one class of $o^-$-type elements, and this class is smaller than every $o^+$-type element.
    \end{enumerate}
  \item for any $f_1,f_2,h\in T_g$ which are $o^+$-type, we have $f_1\precsim f_2\Rightarrow f_1h\precsim f_2h\wedge hf_1\precsim hf_2$. 
  \end{enumerate}

 \end{Prop}

\begin{Rem}\label{Remtgotype}\begin{enumerate}
            \item  Proposition \ref{tgotype}(c) basically says that $T_g$ is the biggest 
 strictly convex subset of $G$ containing $g$ on which $\precsim$ is
    order-type-like.
     \item It follows from Proposition \ref{tgotype}(a) and from Lemma \ref{strictconvexity}(ii) that\newline
     $\min(T_g)=cl(\ig)\cap T_g$.
     \item If $\ig$ is not a welding point, then we can replace ``strictly convex'' by ``convex'' in Proposition \ref{tgotype}(c).
     \item Example \ref{examplesfondamentales}(b) shows that $T_g$ is not always convex.
      \item It is interesting to note that property (d) in \ref{tgotype} is the property satisfied by ordered groups (see axiom (OG) in the introduction).
     \end{enumerate}
\end{Rem}

  We now state the main properties of $G_g$ and $G^g$: 

  \begin{Prop}[Quotient for o-type elements]\label{otypequotient}
   Both $G^g$ and $G_g$ are subgroups of $G$. Moreover, $G^g$ is convex and
   $G_g$ is the smallest normal strictly convex subgroup of $G^g$ such that the q.o induced by 
  $\precsim$ on $G^g/G_g$ is order-type.
 \end{Prop}

\begin{Rem}
   If $\ig$ is not a welding point, then $G_g$ is actually convex. However, Example \ref{examplesfondamentales}(b) shows 
   that $G_g$ is not convex in general. We see that the existence of welding makes the structure of $G$ less smooth, since 
   it prevents the type-components from being convex.
  \end{Rem}

Our goal is now to prove Propositions \ref{tgotype} and \ref{otypequotient}. We start by characterizing the elements of $T_g^+$ in the next two lemmas:

 \begin{Lem}\label{otypeequiv}
  Assume $\ig\precnsim h\precsim g$. Then $\ih\sim \ig$, and in particular $h$ is $o^+$-type.
 \end{Lem}

 \begin{proof}
  By $(CQ_2)$, $h\precsim g$ implies $h\ig\precsim \ig$, hence $h\ig\precsim h$. By $(CQ_2)$  and $(CQ_3)$, this implies
  $\ig\precsim \ih$. Now assume that $\ig\precnsim \ih$ holds. By Lemma \ref{lemfond}, we then have $h\sim h\ig\precsim \ig$, which is a
  contradiction. Therefore, $\ig\sim \ih$.
 \end{proof}

 \begin{Lem}\label{apartenance}
   For any $h\in G$, $h\in T_g^+$ if and only if 
  $h$ is $o^+$-type and $\ih\sim \ig$.  
  In particular, $g\in T_g^+$ and $T_g^-\subseteq cl(\ig)$.
 \end{Lem}

 \begin{proof}
  Assume $h\in T_g^+$.  
  If $\ih\precnsim \ig$, then $\ih\precnsim \ig\precnsim g$. By Lemma \ref{lemfond}, this implies 
  $h\precnsim \ig\precnsim g$, so there is an $o^-$-type element between $h$ and $g$, which is a contradiction.  
  If $\ig\precnsim \ih$, then by the same reasoning we get $g\precnsim \ih\precnsim h$, which is also a contradiction.
  This proves that $\ih\sim \ig$.  Conversely,
  assume that $h$ is $o^+$-type and $\ih\sim \ig$. We want to show that every $f$ between $h$ and $g$ is $o^+$-type.
  Since $f$ is between $h$ and $g$ and since $\ih\sim \ig$, we either have $\ih\precnsim f\precsim h$ or 
  $\ig\precnsim f\precsim g$. By Lemma \ref{otypeequiv}, this implies that $f$ is $o^+$-type.
 \end{proof} 
 
 As a direct consequence of these two lemmas, we have that the q.o is order-type-like on $T_g$:
 
 \begin{Prop}\label{tglookslikeotype}
  $T_g$ contains exactly one class of $o^-$-type elements, which is $T_g^-$. Moreover,
   $T_g^-\precnsim T_g^+$ and there is no $h$ such that $T_g^-\precnsim h\precnsim T_g^+$
 \end{Prop}

 \begin{proof}
  The fact that there is exactly one class of $o^-$type elements is a consequence of Lemma \ref{apartenance}. 
  If $h$ satisfies $T_g^-\precnsim h\precsim T_g^+$, then by Lemma \ref{otypeequiv} $h\in T_g^+$, so we don't have $h\precnsim T_g^+$.
 \end{proof}

 We can now show Proposition \ref{tgotype}:

\begin{proof}[proof of \ref{tgotype}]
 We first prove (a). It is clear from its definition that $T_g^+$ is convex. We also know 
 that $\min(T_g)=T_g^-\subseteq cl(\ig)$ and that there is no element strictly between
 $T_g^-$ and $T_g^+$. It follows that $T_g\cup cl(\ig)$ is convex, which in particular means that 
 $T_g$ is right-convex and that $F_g:=cl(\ig)\backslash T_g$ is the convexity complement of $T_g$ 
 (see Lemma \ref{strictconvexity}). Since $T_g^-\subseteq cl(\ig)$, it follows from the definition of $F_g$ that 
 $cl(\ig)=F_g\cup T_g^-$.
 Now let $h\in F_g$.
 Then $h\sim \ig$.  If $h$ were $o^+$-type, then we would  have $\ih\precnsim\ig\precsim h$. By Lemma \ref{otypeequiv}, this
 would imply that $\ig$ is $o^+$-type, which is a contradiction. Thus, 
 $h$ cannot be $o^+$-type. If $h$ were $o^-$-type, then by Lemma
  \ref{apartenance} we would have $h\in T_g^-$, which is excluded, so $h$ cannot be $o^-$-type. Thus, $h$ must be v-type. It then follows that 
  $F_g=\varnothing$ if and only if $\ig$ is not a welding point, hence (b). 
 Now let us prove (c). It only remains to prove that there is no strictly convex set bigger than $T_g$ satisfying (i) and (ii).
 Towards a contradiction, let $S\varsupsetneq T_g$ be such a set and take $h\in S\backslash T_g$. 
 Assume first that $T_g^+\precnsim h$. Then $h$ must be $o^+$-type and $\ih\precsim T_g^-$.
 Let $g\precsim f\precsim h$. We have 
 $\ih\precnsim f\precsim h$. By Lemma \ref{otypeequiv}, this implies that $f$ is $o^+$-type. Thus, every element between
 $g$ and $h$ is $o^+$-type, so $h\in T_g^+$, which is a contradiction.
 Assume that $h\precsim T_g^-$. Then $h$ must be $o^-$-type and $T_g^+\precnsim \ih$.
 We then have $h\precnsim g\precsim \ih$. By \ref{otypeequiv}, this implies $\ig\sim h$, which 
 means $\ih\in T_g^+$: contradiction. 
 (d) is a direct consequence of Proposition \ref{addition}, since $\ih\nsim f_2$.
\end{proof}

 We mentioned in remark \ref{Remtgotype} that the q.o $\precsim$ on $T_g$ is order-type-like. In fact, the only difference between the structure of $T_g$ and the group in Proposition 
 \ref{ordertypegroup} is that 
 $\precsim$ is not an order on $T_g^+$ (see for example 
 $T_h^+$ in Example \ref{componentsexample}). However, we have the following, which will be useful in the proof of Proposition \ref{otypequotient}:
 
 \begin{Lem}\label{ordretgplus}
  Let $f,h\in T_g^+$ and $f\sim h$. Then $f\ih\in G_g$.
 \end{Lem}

 \begin{proof}
  By simply applying $(CQ_2)$ to the inequalities $f\precsim h\precsim f$
  we obtain $f\ih\precsim \ih\sim \ig$ and $hf^{-1}\precsim f^{-1}\sim \ig$, hence $f\ih\in G_g$.
 \end{proof}
 
Intuitively, we see from Lemma \ref{ordretgplus} that the q.o induced by $\precsim$ on the quotient $G^g/G_g$ will satisfy the condition of 
Proposition \ref{ordertypegroup}, where the sets $G^-$ and $G^+$ of Proposition \ref{ordertypegroup} will respectively correspond to
 $T_g^-/G_g:=\{hG_g\mid h\in T_g^-\}$ and $T^+/G_g:=\{hG_g\mid h\in T_g^+\}$. The next two propositions will help us prove
  Proposition  \ref{otypequotient}:
 
 \begin{Prop}\label{Ggstrictconvex}
 We have  $G_g=\{h\precnsim \ig\}\cup F_g$, where $F_g$ is as in Proposition \ref{tgotype}. In particular, $G_g$
 is left-convex, and it is convex if and only if $F_g=\varnothing$. This in turn holds
  if and only if $T_g$ is convex. Moreover, if $G_g$ is not convex, then its convexity complement is $T_g^-$ and we have 
  $F_g=\max(G_g)$.
 \end{Prop}

 \begin{proof}
  Let $h\in G_g$. Then in particular $\{h,\ih\}\precsim \ig$. If $h\sim \ig$, then we have  $h\in cl(\ig)=T_g^-\cup F_g$. Since 
  $\ih\precsim \ig$, we cannot have $h\in T_g^-$ (otherwise we would have $\ih\in T_g^+$), so $h\in F_g$.  
  Conversely, assume $h\precnsim \ig$. By Lemma \ref{lemfond}(ii), this implies 
  $\ih\precnsim \ig$, hence $h\in G_g$. Assume $h\in F_g$. Then $h\sim \ig$ and $h$ is v-type, so $\ih\sim h\precsim \ig$, hence
  $h\in G_g$.  
  If $F_g=\varnothing$, then $G_g=\{h\precnsim \ig\}$ is clearly convex.
  Now assume that $F_g\neq\varnothing$. By definition of $F_g$, we have 
  $f\sim\ig$ for every $f\in F_g$, hence  
   $h\precsim f$ for every $f\in F_g$ and $h\in G_g$. Since 
  $F_g\subseteq G_g$, it follows that $F_g=\max(G_g)$. Now take any $f\in F_g$.
  We then have $f\in G_g$, $f\sim \ig$, but $\ig\notin G_g$, so $G_g$ is not convex. Moreover, 
  we have
  $T_g^-=cl(f)\backslash G_g$. By Lemma \ref{strictconvexity}(iii), this implies that $T_g^-$ is the convexity complement of $G_g$. 
 \end{proof}

 \begin{Prop}\label{Gginitialsegment}
  $G^g$ is an initial segment of $G$.
 \end{Prop}

 \begin{proof}
  We already showed that $G_g\cup T_g^-$ is an initial segment. Since $T_g^+$ is convex and since there is no element
  strictly contained between $T_g^-$ and $T_g^+$, it follows that $G^g=G_g\cup T_g^-\cup T_g^+$ is an initial segment.
  
 \end{proof}

 We can now show Proposition \ref{otypequotient}:

 \begin{proof}[proof of \ref{otypequotient}]
 Let $h_1,h_2\in G_g$. We have $h_1\precsim \ig$ and $\{h_2,\ih_2\}\precsim \ig\precnsim g$, so 
  we can apply Propositions \ref{addition} and \ref{toutesequivalences} and get
 $h_1\ih_2\precsim \ig\ih_2\sim \ig$. By a similar argument, we also have 
  $h_2\ih_1\precsim \ig$, hence $h_1\ih_2\in G_g$. This proves that $G_g$ is a subgroup of $G$. Now let us show that $G^g$ is a subgroup of $G$. 
Note that by Propositions \ref{tglookslikeotype}, \ref{Ggstrictconvex} and \ref{Gginitialsegment} we have 
  $G_g\precsim T_g^-\precnsim T_g^+$. Since $G^g$ is moreover an initial segment of $G^g$, it follows that an element $h\in G$ is in $G^g$ if and only if there exists $f\in T_g^+$ with 
  $h\precsim f$. 
  Let $h_1,h_2\in G^g$. There exists $f\in T_g^+$ with $\{h_1,h_2\}\precsim f$.  
  Assume $h_2\sim f$. By convexity of $T_g^+$, this implies that $h_2\in T_g^+$. We then have 
  $h_1\precsim h_2$. By $(CQ_2)$, this implies $h_1\ih_2\precsim \ih_2\precnsim h_2\in T_g^+$, hence $h_1\ih_2\in G^g$.  
  Assume $h_2\precnsim f$. By Proposition \ref{addition}, this implies $h_1\ih_2\precsim f\ih_2$. If $f\ih_2\precsim f^{-1}$,
  then $h_1\ih_2\precsim f\in T_g^+$, hence $h_1\ih_2\in G^g$. Assume then that 
  $f^{-1}\precnsim f\ih_2$.
  By Lemma \ref{lemfond} (i), $h_2\precnsim f$ implies  $f^{-1}\sim h_2f^{-1}$, hence $h_2f^{-1}\precnsim f\ih_2$ which means that $f\ih_2$ is $o^+$-type. Since $f\in T_g^+$, we have $f^{-1}\sim \ig$ by Lemma \ref{apartenance}. 
  We thus have $\ig\sim h_2f^{-1}$ and $f\ih_2$ is $o^+$-type. By Lemma \ref{apartenance}, we then have $f\ih_2\in T_g^+$. Since $h_1\ih_2\precsim f\ih_2$, it follows that $h_1\ih_2\in G^g$. This proves that $G^g$ is a 
subgroup of $G$.
  Now let us show that $G_g$ is normal in $G^g$. Let $h\in G_g$ and $z\in G^g$.
  By $(CQ_3)$, we have $\{(\ih)^z,h^z\}\precsim (\ig)^z$. It is enough to show that 
  $(\ig)^z\precsim\ig$. Note that by $(CQ_3)$, conjugation preserves types, so $(\ig)^z$ is $o^-$-type.
  Since $G^g$ is a group, we have $(\ig)^z\in G^g$, and since there is no $o^-$-type element above $\ig$ in $G^g$ we must
  have $(\ig)^z\precsim\ig$.

   Now let us prove that the q.o induced on $G^g/G_g$ is order-type. Set $G^+:=T_g^+/G_g$ and $G^-=T_g^-/G_g$. Clearly, $G^-=(G^+)^{-1}$.   
   Remember that, if $f\notin G_g$, then
   $h\precsim f\Leftrightarrow hG_g\precsim fG_g$. Since $T_g^-\precnsim T_g^+$, we have
   $G^-\precnsim G^+$.
   Let $fG_g,hG_g\in G^-$. Since $\precsim$ is trivial on $T_g^-$, we have 
   $h\sim f$, which implies $hG_g\sim fG_g$.   
  Moreover, Lemma \ref{ordretgplus} implies that $\precsim$ is an order on $G^+$.    
   By proposition \ref{ordertypegroup}, $\precsim$ is order-type on $G^g/G_g$.   
   Now assume that $H\varsubsetneq G_g$ is another strictly convex normal subgroup of $G^g$. Take 
    $h\in G_g\backslash H$. If $G_g$ is convex, then we have $h\precnsim T_g^-$, so 
    $H\precnsim hH\precnsim G^-$. By Proposition \ref{ordertypegroup}, it follows that  $\precsim$ on $G^g/H$ cannot be order-type.
    If $G_g$ is not convex, then we can choose $h\in F_g$ so $h$ is v-type and so is 
    $hH$, so $\precsim$ on $G^g/H$ is also not order-type.    
 \end{proof}

 \subsubsection{ $T_g$ in the v-type case}
 
   Assume now that $g\neq1$ is v-type.
   
   \begin{Lem}
    $g\in T_g$. 
   \end{Lem}
  \begin{proof}
   All we have to check is that there is no $o^+$-type element equivalent to $g$.
   This is given by Lemma \ref{otypeequiv}   
  \end{proof}

\begin{Lem}
 Let $h$ be o-type. Then either $h\precnsim T_g$ or $T_g\precsim h$.
\end{Lem}
\begin{proof}

 By Proposition \ref{tgotype}, $T_h$ is right-convex and contains $h$. Moreover, $T_h$ only contains $o$-type elements and $T_g$ only 
 contains $v$-type elements, hence the claim.
\end{proof}

  \begin{Prop}\label{tgvtypeconvex}
   Define $F_g:=\{h\quad o^-\text{-type }\mid h\sim \max(T_g)\}$ if $\max T_g\neq\varnothing$ and $F_g:=\varnothing$ otherwise.
   Then $T_g$ is left-convex with convexity complement $F_g$. In particular, $T_g$ is convex if 
it has no maximum.
  \end{Prop}

  \begin{proof}
  Assume $T_g$ is not convex. Then there exists $h_1,h_2\in T_g$ and $f\notin T_g$ such that 
   $h_1\precsim f\precsim h_2$. If $f$ were $v$-type, then since $f\notin T_g$ there would an 
  $o^+$-type element between $g$ and $f$. This would imply that there is an $o^+$-type element either between 
  $g$ and $h_1$ or between $g$ and $h_2$, which is a contradiction. For the same reason $f$ cannot be $o^+$-type.  
  Thus, $f$ is $o^-$-type. It follows from the previous lemma that $T_g\precsim f$, so $h_2\precsim f$, hence
  $h_2\sim f$. It follows that $h_2\in \max(T_g)$. 
  Now let us show that $T_g\cup cl(h_2)$ is convex. Let $f_1,f_2\in T_g\cup cl(h_2)$ and $f_1\precsim f\precsim f_2$. With the same
  reasoning as above, $f$ cannot be $o^+$-type so it must either be v-type or $o^-$-type. If it is v-type, then 
  $f\in T_g$. If it is $o^-$-type, then $f\sim h_2$.
  \end{proof}

  We can now state a v-type analogue of Proposition \ref{tgotype}:
  
  \begin{Prop}\label{tgvtype}
   The set $T_g$ is the biggest  strictly convex subset of $G\backslash\{1\}$ containing $g$ such that every element of $T_g$ is v-type.
   If $G$ has no welding at $g$, then $T_g$ is even convex.
  \end{Prop}

  \begin{proof}
   Let $S\varsupsetneq T_g$ be strictly convex and let $h\in S\backslash T_g$ with $h\neq1$ be v-type. 
   Since $h\notin T_g$, then by definition of $T_g$ there must be an $o^+$-type element $f$ between
   $g$ and $h$.  By Lemma \ref{otypeequiv}, we have $f\nsim h$ and $f\nsim g$, so 
   $f$ is strictly between $g$ and $h$, hence $f\in S$. Thus, $S$ must contain o-type elements.
  \end{proof}

   We now want to establish the v-type analogue of Proposition \ref{otypequotient}.
  
  \begin{Prop}\label{Ggvtypecase}
  Both $G^g$ and $G_g$ are  subgroups of $G$,
   $G^g$ is  strictly convex with convexity complement $F_g$ and
   $G_g$ is convex. Moreover, $G_g$ is normal in $G^g$.  
  \end{Prop}

  \begin{proof}
    $G_g$ is clearly an initial segment by definition, so it is convex.
   Moreover, we know that $T_g$ is left-convex and that there is no element strictly contained
   between $G_g$ and $T_g$, so it follows immediately that $G^g=G_g\cup T_g$ is left-convex. We also know 
   that $F_g$ is the convexity complement of $T_g$ and so it is also the convexity complement of $G^g$.
 
   Let us show that $G_g$ is a group.
   Let $f_1,f_2\in G_g$ and $h\in T_g$, so in particular $h$ is v-type.
   Assume $h\precsim f_1f^{-1}_2$. We then have $f_1\nsim f_1f^{-1}_2$. Applying Proposition 
    \ref{addition}, we get $f^{-1}_1h\precsim f^{-1}_2$. However, by Proposition \ref{toutesequivalences}, 
    we have $f^{-1}_2\precnsim h$ and $f^{-1}_1h\sim h$, so this is a contradiction. Thus, we must have 
    $f_1f_2^{-1}\precnsim h$. Since $h$ is arbitrary in $T_g$, this means $f_1f^{-1}_2\in G_g$.
   Now let us show that $G^g$ is a group. Let $f_1,f_2\in G^g$. This implies that there is $h\in T_g$ with $\{f_1,f_2\}\precsim h$. 
   If $h\sim f_2$, then $f_2\in T_g$, so $f_2$ is v-type and we have $f_1\precsim f_2$. By 
   $(CQ_2)$, it follows that $f_1f_2^{-1}\precsim f_2^{-1}\sim h\in T_g$. If $h\nsim f_2$, then $f_1\precsim h$ implies $f_1f_2^{-1}\precsim hf_2^{-1}$ by
   Proposition \ref{addition}. Since $h$ is v-type, $f_2\precnsim h$ implies $h\sim hf_2^{-1}$ by Proposition \ref{toutesequivalences}, hence $f_1f_2^{-1}\precsim  h$. In any case 
   we have $f_1f^{-1}_2\precsim h$, which means $f_1f^{-1}_2\in G^g\cup F_g$. We can  show with the same 
   reasoning that $f_2f_1^{-1}\precsim h$. This implies that $f_1f^{-1}_2\notin F_g$. 
   Indeed, if 
   $f_1f^{-1}_2$ were in $F_g$, then it would be $o^-$-type, so we would have $h\sim f_1f^{-1}_2\precnsim f_2f^{-1}_1$.
   
   Take $h\in G_g$ and $z\in T_g$. If $h\neq 1$, then  there exists an $o^+$-type element
   $f$ between $h$ and $g$. We then have
   $h^z\precsim f^z\precsim g^z\in G^g$, so there is an $o^+$-type element between $h^z$ and 
   $g^z$, hence $h^z\in G_g$.
  \end{proof}

  \begin{Prop}\label{vtypequotient}
   The group $G_g$ is the smallest normal convex subgroup of $G^g$ such that the quotient $G^g/G_g$ is valuational.
  \end{Prop}

   \begin{proof}
    Remember that for any $f,h\notin G_g$, $f\precsim h$ if and only if $fG_g\precsim hG_g$. Since every element of $T_g$
    is v-type, it follows that every element of $G^g/G_g$ is also v-type, so the q.o is valuational.    
    If $H$ is strictly contained in $G_g$, then $G^g\backslash H$ contains an $o$-type element $h$, and then 
    $hH$ is o-type.
   \end{proof}

   \begin{Rem}
   \begin{enumerate}
    \item As happens in the o-type case, welding is the only thing preventing $T_g$ and $G^g$ from being convex. If 
    $G$ has no welding point, then we can replace ``strictly convex'' by ``convex'' in Propositions 
    \ref{tgvtype} and \ref{vtypequotient}.
    \item In the o-type case as well as in the v-type case, it can happen that $G_g$ and $G^g$ are not normal in $G$
    (see Example \ref{exempletypeval} below).    
    \end{enumerate}
    
   \end{Rem}

   \subsubsection{Type-valuation}
   
    We can now show that the $T_g$'s form a partition of $G$:
 
 \begin{Prop}\label{Tgpartition}
  The following  holds for any $g,h\in G$:\newline
  $g\in T_h\Leftrightarrow h\in T_g\Leftrightarrow T_g=T_h\Leftrightarrow T_g\cap T_h\neq\varnothing\Leftrightarrow  
  G_g=G_h\Leftrightarrow G^g=G^h$.
 \end{Prop}

 \begin{proof}
 
  Assume $g\in T_h$. If $h,g$ are v-type, then
  we use Proposition \ref{tgvtype}.
   We know that $T_g$ is the biggest strictly convex subset of $G$ containing $g$ whose every
   element is v-type. Since  $T_h$ is strictly convex and only contains v-type elements and $g\in T_h$, it follows that 
   $T_h\subseteq T_g$. This implies $h\in T_g$. By a similar argument, it also follows that $T_g\subseteq T_h$, hence
   $T_g=T_h$.
   The case where they are o-type is similar by using Proposition \ref{tgotype}. This proves the first two
   equivalences. The third one follows immediately: if $T_g\cap T_h\neq\varnothing$, then there is 
   $f\in G$ with $f\in T_g\cap T_h$, which implies $T_g=T_f=T_h$.   
   Assume $T_g=T_h$. In the v-type case we obviously have $G_h=G_g$ by definition of $G_g$.
   If they are $o^+$-type, then $\ig\sim \ih$.  But then, for any $f\in G$, 
   $\{f,f^{-1}\}\precsim \ig$ is equivalent to $\{f,f^{-1}\}\precsim \ih$, hence $G_g=G_h$.
   Assume $G_g=G_h$. Without loss of generality $g\precsim h$. Since $G_g=G_h$, we have $g\notin G_h$. Note that there is no 
   element strictly contained between $G_h$ and $T_h$ (otherwise,
   there would be an element $f$ with $f\notin G^h=G_h\cup T_h$ and $1\precnsim f\precnsim h$. This would 
   contradict the fact that $G^h$ is strictly convex). Thus,
we  have $g\in T_h$, hence $T_h=T_g$, which also 
   implies $G^g=G^h$. Finally, assume $G^g=G^h$. By definition of $G^g,G^h$, this implies that 
   $T_h\cap T_g\neq\varnothing$, hence 
   $T_g=T_h$.
 \end{proof} 
 
  We have thus reached the goal we announced in the introduction: we showed that  $G$ is partitioned into a family of sets 
  on each of which the C-q.o is elementary-type-like. Our next objective is to reformulate this statement by 
  showing that $\precsim$ can be obtained by lifting elementary C-q.o's. To do this we need to define a valuation on $G$
  whose fibers are the type-components. We first notice that $\precsim$ naturally induces an order on the set of 
  type-components:
  
 \begin{Prop}\label{ordertypecomponents}
  Define $\leq$ on the set of all type-components by 
  $T_g\leq T_h\Leftrightarrow T_g=T_h\vee T_g\precsim T_h$. This is an order on the set of 
  all type-components of $G$.  
 \end{Prop}

 \begin{proof}
  The fact that $\leq$ is total follows from the fact that the type-components are strictly convex and pairwise disjoint. 
The relation $\leq$ is clearly reflexive and transitive, let us prove that is is antisymmetric.
 If $T_g\precsim T_h\precsim T_g$, then  all elements of $T_g\cup T_h$ are equivalent to one another. It follows that 
 $h,g$ must both be v-type. Since $g\sim h$, this implies $T_g=T_h$. 
 \end{proof}

 \begin{Rem}
  If $S$ is a subset of $G$ which contains elements $s,t\in S$ such that $s\precnsim t$, then $S\precsim S$ does not hold
  (remember that $S\precsim T$ means that $s\precsim t$ for \textit{any} pair 
  $(s,t)\in S\times T$). Hence the condition $T_g=T_h$ does not imply $T_g\precsim T_h$. 
  Therefore, the condition ``$T_g=T_h$'' in the definition of $\leq$ is  essential for 
reflexivity.
   \end{Rem}

 \begin{Prop}\label{typevaluation}
  Set $\Gamma:=\{T_g\mid g\in G\}$ and let $\leq^{\ast}$ be the reverse order of the one given in Proposition 
  \ref{ordertypecomponents}.
  We define a valuation on $G$ called the \textbf{type-valuation associated to $\precsim$} by 
  \begin{align*} 
   v:\quad & G\to(\Gamma,\leq^{\ast})\\
   &g\mapsto  T_g.
  \end{align*}

 \end{Prop}
\begin{proof}
 Clearly, $T_1$ is a maximum of $(\Gamma,\leq^{\ast})$ and $v(g)=v(\ig)$ for any $g\in G$. Let $g,h\in G$ with 
 $v(g)\leq^{\ast} v(h)$. By definition of $\leq^{\ast}$, it follows that $h\in G^g$. Since $G^g$ is a group, we then have 
 $gh\in G^g$. This implies $T_{gh}=T_g$ or $T_{gh}\precsim T_g$, which means $v(g)\leq^{\ast} v(gh)$, hence 
 $\min(v(g),v(h))\leq^{\ast}v(gh)$. Now let $z\in G$. If $T_h\precsim T_g$, then in particular 
 $h\precsim g$, so $h^z\precsim g^z$. This implies $v(g^z)\leq^{\ast} v(h^z)$. Now assume $T_g=T_h$. If $g,h$ are both 
 v-type, then so are $g^z$ and $h^z$ (this follows from $(CQ_3)$). Since $h\in T_g$, there is no $o^+$-type element between $g$ and $h$. Therefore, 
 by $(CQ_3)$, there cannot be an $o^+$-type elements between $g^z$ and $h^z$. This proves $T_{g^z}=T_{h^z}$. 
 The same kind of argument show $T_g^+=T_h^+$ in the case where $g,h$ are both $o^+$-type. If one of them is 
 $o^-$-type, then take their inverse and we are back to the $o^+$-type case. 
\end{proof}

 \subsection{Structure theorems}\label{maintheoremsection}
   
   We now want to summarize the results of Section \ref{typecomponents} into a structure theorem of 
   C-q.o.g's.   
   We start by giving two ways of constructing C-q.o's: lifting from quotients and ``welding''. This will justify 
   the fact that the q.o's given in Example \ref{examplesfondamentales} are indeed C-q.o's. We then show
   that any C-q.o can be obtained by lifting C-q.o's of elementary type and then welding if necessary.

\begin{Prop}[construction by lifting]\label{liftingweak}
  Let $G$ be a group, $v:G\to \Gamma\cup\{\infty\}$ a valuation. Assume that for each $\gamma$, 
  the quotient $G^{\gamma}/G_{\gamma}$ is endowed with a C-q.o $\precsim_{\gamma}$. Assume moreover that for any 
  $z\in G$ and any $\gamma\in\Gamma$, the isomorphism $G^{\gamma}/G_{\gamma}\to G^{\gamma^z}/G_{\gamma^z}$ 
  induced by conjugation by $z$ is quasi-order-preserving. Then the lifting of 
  $(\precsim_{\gamma})_{\gamma\in\Gamma}$ to $G$ is also a C-q.o.
 \end{Prop}

 \begin{proof}
  Denote by $\precsim$ the lifting.  
  $(CQ_1)$ is clearly satisfied.   
  Let $x\precsim y$. If $v(x)>v(y)$, then $v(x\iy)=v(\iy)=:\gamma$ and
  $x\iy G_{\gamma}=\iy G_{\gamma}$. This implies $x\iy G_{\gamma}\precsim_{\gamma} \iy G_{\gamma}$, hence $x\iy\precsim \iy$. We also have 
  $v(x^z)>v(y^z)$, hence $x^z\precsim y^z$.  
  Assume $v(x)=v(y)=\gamma$ and $xG_{\gamma}\precsim_{\gamma}yG_{\gamma}$. 
 This implies $v(x^z)=v(y^z)=\gamma^z$. By assumption,
  $x^zG_{\gamma^z}\precsim_{\gamma} y^zG_{\gamma^z}$, hence $x^z\precsim y^z$. Moreover, we have 
 $v(x\iy)\geq\min(v(x),v(y))=\gamma$. If $v(x\iy)>\gamma$, then $x\iy\precsim\iy$, so assume $v(x\iy)=\gamma$.
  Since $\precsim_{\gamma}$ is a C-q.o, we have $x\iy G_{\gamma}\precsim_{\gamma} \iy G_{\gamma}$, hence $x\iy\precsim \iy$. 
  
 \end{proof}
 
 As a special case of lifting we can define a C-q.o on semi-direct products, which is how we obtained 
 Example \ref{examplesfondamentales}(e):
 \begin{Prop}\label{semidirectproductlift}
  Let $(G,\precsim_G),(H,\precsim_H)$ be two C-q.o.g and let 
  $\alpha: G\to Aut(H)$ such that for any $g\in G$, $\alpha(g)$ preserves $\precsim_H$. Define 
  a q.o $\precsim$ on $G\ltimes_{\alpha}H$ by \newline  
  $(g_1,h_1)\precsim (g_2,h_2)\Leftrightarrow (g_1\precsim_G g_2)\wedge(g_2\neq 1\vee(g_2=1\wedge h_1\precsim_H h_2))$.
  Then $\precsim$ is a C-q.o.
 \end{Prop}
\begin{proof}

 Set $F:=G\ltimes_{\alpha}H$, $\Gamma:=\{1,2\}$ and define $v: F\to\Gamma\cup\{\infty\}$ as follows:\newline 
 $v(g,h):=\begin{cases}
           1 \text{ if }g\neq1.\\
           2 \text{ if } g=1\neq h.\\
           \infty\text{ if } g=h=1.
          \end{cases}$
          
          This defines a valuation on $F$. 
          We have $F_2\cong\{1\}$, $F^2=F_1= \{1\}\times H$ and 
          $F^1=F$. Now take $h_1,h_2\in H\cong F^2/F_2$ with $h_1\precsim_H h_2$ and $z=(g,h)\in F$.
          Since $H$ is normal in $F$, we have $F^2/F_2=F^{2^z}/F_{2^z}$.
          We have $h_i^z=(\alpha(g)(h_i))^h$ for $i=1,2$. 
          By assumption, $\alpha(g)$ preserves $\precsim_H$, hence $\alpha(g)(h_1)\precsim_H\alpha(g)(h_2)$.
          By $(CQ_3)$, it then follows that
          $h_1^z\precsim_H h_2^z$. This proves that the isomorphism $F^2/F_2\to F^{2^z}/F_{2^z}$  induced by $z$ preserves 
          $\precsim_H$. Now note that 
          $G\cong F^1/F_1$, so $F^1/F_1$ is endowed with a C-q.o. $\precsim_G$ defined by \newline
          $(g_1,1).F_1\precsim_G(g_2,1).F_1\Leftrightarrow g_1\precsim_Gg_2$.          
          Take $(g_1,1).F_1$ and $(g_2,1).F_1$ in  $F^1/F_1$ with $(g_1,1).F_1\precsim_G(g_2,1).F_1$. By definition, 
          $((g_i,1).F_1)^z=(g_i,1)^z.F_1=(g_i^g,h\alpha(g_i^g)(\ih)).F_1$\newline $=(g_i^g,1).F_1$. 
          Because $(g_1,1).F_1\precsim_G(g_2,1).F_1$, it follows from $(CQ_3)$ on $G$ that \newline
          $(g_1^g,1)\precsim_G(g_2^g,1)$, hence $((g_1,1).F_1)^z\precsim_G((g_2,1).F_1)^z$.
          This proves that the isomorphism 
          $F^1/F_1\to F^{1^z}/F_{1^z}$ induced by $z$ preserves $\precsim_G$. Thus, the hypothesis of 
          Proposition \ref{liftingweak} are satisfied, so the lifting of 
          $(\precsim_H,\precsim_G)$ to $F$ is a C-q.o.

\end{proof}

  We now introduce another way of obtaining C-q.o's, which we call welding. 
Let $g$ be an o$^-$-type element, and assume that  the maximum $M_g$ of $G_g$ is non-empty. 
We noted in Proposition \ref{Ggstrictconvex} that, if $F_g\neq\varnothing$, then $M_g=F_g$, and so, by 
Proposition \ref{tgotype}(a), we have $M_g\subseteq cl(g)$. If $F_g=\varnothing$, then by Proposition 
\ref{Ggstrictconvex} we have $G_g=\{h\in G\mid h\precnsim g\}$. In any case, there is no element
strictly contained between $M_g$ and $cl(g)$.
 This means that we can coarsen $\precsim$ by joining the sets $cl(g)$ and $M_g$. In other words, we define a coarsening 
$\precsim_2$ of $\precsim$ by declaring that $h\sim_2f$ for any $f,h\in M_g\cup cl(g)$ and $h\precsim_2f\Leftrightarrow h\precsim f$ whenever 
 $h\notin M_g\cup cl(g)$ or $f\notin M_g\cup cl(g)$. Note that, in example \ref{examplesfondamentales}(b), if we set 
 $g:=(-1,0)$, then we have $G_g=\{0\}\times\Z$ and $M_g=\{0\}\times(\Z\backslash\{0\})\subseteq cl(g)$. Therefore, 
 it can happen that $M_g\subseteq cl(g)$, in which case nothing changes. But if 
$T_g$ is convex, then by \ref{Ggstrictconvex} we have $M_g\cap cl(g)=\varnothing$, and then $\precsim_2$ is different from $\precsim$. If we
 apply this coarsening operation simultaneously at each $g^z$ for $z\in G$, then we will obtain a new C-q.o,
as the next proposition shows:
  
 \begin{Prop}[Construction by welding]\label{weldingconstruction}
  Let $(G,\precsim)$ be a C-q.o.g and $g\in G$ an $o^-$-type element such that 
  $M_g:=\max(G_g)$ is non-empty. Then for any $z\in G$, 
  $M_{g^z}:=\max(G_{g^z})$ is also non-empty, so we can define a coarsening $\precsim_2$ of $\precsim$ by declaring 
  $M_{g^z}\sim_2 g^z$ for every $z\in G$. Moreover, this coarsening is a C-q.o.
 \end{Prop}
 
 \begin{proof}
  Note that by $(CQ_3)$, we have $g\precnsim \ig\Rightarrow g^z\precnsim (\ig)^z=(g^z)^{-1}$, so 
  $g^z$ is o$^{-}$-type. The fact that $M_{g^z}$ is non-empty is also a direct consequence of $(CQ_3)$. 
  It also follows from $(CQ_3)$ that $F_g\neq\varnothing\Leftrightarrow F_{g^z}\neq\varnothing$. Note also that 
  if $F_g\neq\varnothing$, then by Proposition \ref{Ggstrictconvex} we have 
  $M_g=F_g$, so we already have $M_g\sim g$. By $(CQ_3)$, this implies $M_{g^z}\sim g^z$ for all $z\in G$. 
  It then follows that $\precsim=\precsim_2$, so there is nothing to prove. Therefore, we can assume 
  without loss of generality that $F_{g^z}=\varnothing$ for all $z\in G$.

  Set $\precsim_1:=\precsim$. We want to show that $\precsim_2$ is a C-q.o. 
  Let $x,y,z\in G$ with $x\precsim_2 y$. 
  If $x\precsim_1y$, then we have $x\iy\precsim_1\iy$ and 
  $x^z\precsim_1y^z$. Since $\precsim_2$ is a coarsening of $\precsim_1$, this implies 
  $x\iy\precsim_2\iy$ and $x^z\precsim_2y^z$. Now assume $y\precnsim_1x$. This can only happen if there is 
  $w\in G$ with $y\in M_{g^w}$ and $x\sim_1 g^w$. Since we assumed that $F_{g^w}=\varnothing$, 
  it follows that $x$ is o$^-$-type.
  By maximality of $y$, we have $\iy\precsim_1 y$. We thus have 
  $\{y,\iy\}\precsim_1 x\precnsim_1\ix$.
 By Lemma 
 \ref{lemfond}(iii), this implies $x\iy\sim_1 x$. 
 By $(CQ_2')$, $\iy\precnsim_1 y$  would imply $y\precnsim_1 y^2$, which would contradict the maximality of $y$. It follows that $y$ is v-type.
  We thus have $x\iy\sim_1 g^w$ and $\iy\in M_{g^w}$.
 By definition of 
  $\precsim_2$, this implies $x\iy\sim_2\iy$. Moreover, we have $y^z\in M_{g^{wz}}$ and $x^z\sim_1 g^{wz}$, which also 
  implies $x^z\sim_2 y^z$.
 \end{proof}
 
 We see that, if we lift a family of C-q.o's of elementary types as in Proposition 
 \ref{liftingweak} and then apply welding, then the q.o which we obtain is again 
 a C-q.o. Our main theorem states that any C-q.o
 is obtained through this process:

 \begin{Thm}[Structure theorem of a C-q.o.g]\label{structuretheorem}
  Let $(G,\precsim)$ be a C-q.o.g.  
  There exists a valuation $v$ on $G$ with value set $\Gamma\cup\{\infty\}$, called the type-valuation associated to 
  $\precsim$, such that 
  the following holds:\begin{enumerate}[(i)]
                                \item For any $\gamma\in\Gamma$, 
                                $G^{\gamma}$ and $G_{\gamma}$ are $\precsim$-strictly-convex subgroups of $G$.                              
                                \item The q.o $\precsim_{\gamma}$ induced by $\precsim$ on 
                                $H_{\gamma}:=G^{\gamma}/G_{\gamma}$ is of elementary type.
                                \item If $\gamma\leq \delta$, if $\precsim_{\gamma},\precsim_{\delta}$ are both valuational, 
 then there exists $\alpha$ between $\gamma$ and $\delta$ such that
 $\precsim_{\alpha}$ is order-type.
                               \end{enumerate}
                               
 Moreover, the q.o $\precsim$ can be obtained by lifting the family $(\precsim_{\gamma})_{\gamma\in\Gamma}$ to 
 $G$ and then welding if necessary.
 \end{Thm}

 \begin{proof}
  We already defined the type-valuation $v$ in Proposition \ref{typevaluation}. Note that for any $g\in G$, we have 
  $G^{v(g)}=G^g$ and $G_{v(g)}=G_g$. (i) and (ii) follow from Propositions \ref{otypequotient}, \ref{Ggvtypecase} and \ref{vtypequotient},
  (iii) follows from \ref{tgvtype}.
  Denote by $\precsim^{\ast}$ the lifting of $(\precsim_{\gamma})_{\gamma\in\Gamma}$ to $G$. Note that an element $g\in G$ 
  is v-type (respectively, o$^-$-type) with respect to $\precsim$ if and only if it is v-type 
  (respectively, o$^-$-type) with respect to $\etoile$ (this follows easily from Propositions \ref{otypequotient} and \ref{vtypequotient} and from 
  the definition of the the lifting).
   We first
  show that $\precsim$ is a coarsening of $\precsim^{\ast}$. Let $g,h\in G$ with 
  $g\precsim^{\ast}h$. By definition of $\precsim^{\ast}$, we either have 
  $v(h)>v(g)$ or $v(g)=v(h)\wedge gG_g\precsim_{v(g)} hG_g$. In the first case we have 
  by definition of $v$:  $g\precsim h$. In the second case, since $G_h=G_g$, we have 
  $h\notin G_g$. Thus, by definition of the q.o induced on the quotient, we must have 
  $g\precsim h$. This proves  that $\precsim$ is a coarsening of $\precsim^{\ast}$.  
  Now let $g,h\in G$ be such that $g\precsim h$ but $h\precnsim^{\ast}g$.  We will show that 
  $h$ is v-type, $g$ is o$^-$-type and $h\in\max(G_g,\etoile)$. It will then follow that 
  $\precsim$ is obtained from $\etoile$ by welding $g$ and $\max(G_g,\etoile)$.
  By definition of 
  $\precsim^{\ast}$, $h\netoile g$ means either $v(h)>v(g)$ or $v(g)=v(h)$ and 
  $hG_g\precnsim_{v(g)} gG_g$. But the latter case would imply $h\precnsim g$, so we must have 
  $v(h)>v(g)$ i.e $h\in G_g$. This implies $h\precsim g$, so $g\sim h$.
  If $h$ were o-type, then by Proposition \ref{otypequotient} $G^h$ would be convex with respect to $\precsim$. The inequality
  $g\precsim h$ would then imply $g\in G^h$, which contradicts $v(g)<v(h)$. Therefore, $h$ must be 
  v-type. Assume for a contradiction that $g$ is v-type.  Since $v(h)>v(g)$, we have $g\notin T_h$. By 
  definition of $T_h$, it follows that there is an o$^+$-type element $f$ between $g$ and $h$. 
  But since $g\sim h$, it follows that $f\sim h$. This contradicts Lemma \ref{otypeequiv}.
   Therefore, $g$ is o-type. 
  Since $h\sim g$, $h$ is in the convexity complement of $T_g$. By Proposition 
  \ref{Ggstrictconvex}, we thus have  $h\in\max(G_g,\precsim)$.
   Now let $f\in G_g$ with $h\precsim^{\ast} f$. Since $\precsim$ is a coarsening of 
  $\precsim^{\ast}$, we then have $h\precsim f$, hence $h\sim f$ by maximality of $h$. 
   Now $h\precsim^{\ast} f$ implies $v(f)\leq v(h)$ and $f\in G_g$ implies $v(f)>v(g)$. Since $h\in\max (G_g,\precsim)$, there is no element strictly contained between 
  $T_h$ and $T_g$, so we must have $v(h)=v(f)$. Since $h$ is v-type, $T_h$ is left-convex, so  $h\sim f$ implies $f\notin G_h$. 
  By definition of $\precsim_{v(h)}$ (see Proposition \ref{quotientconvex}), it then follows that 
  $hG_h\sim_{v(h)}fG_h$, hence $h\sim^{\ast}f$. This shows that $h$ is maximal in $(G_g,\precsim^{\ast})$.
Thus, the only point on which $\precsim$ 
  and $\precsim^{\ast}$ disagree are welding points, so $\precsim$ is obtained from $\precsim^{\ast}$ by welding.
 \end{proof}

 \begin{Rem}
  If there is no welding point, then $\precsim$ actually coincides with the lifting of $(\precsim_{\gamma})_{\gamma\in\Gamma}$.
  
 \end{Rem}

 \begin{Ex}\label{exempletypeval}
  We take notations from Examples \ref{examplesfondamentales}. We are going to give an explicit definition to the type-valuation 
 associated to the C-q.o's $\precsim_H$ and $\precsim_F$ of examples (d) and (e). We already defined a valuation 
 $v_G:G\to\{1,2\}\cup\{\infty\}$ on $G$ and a valuation $w_H:H\to\Z$ on $H$.
 Define $\Gamma:=\Z\times\{1,2\}$ and order $\Gamma$ lexicographically, i.e $(x,y)\leq (x',y')\Leftrightarrow (x<x'\vee (x=x'\wedge y\leq y'))$. Define 
$v_H:H\to\Gamma$ by $v(\sum_{n\in\Z}g_n\tau_n):=(k,v_G(g_k))$, where $k=w_H(\sum_{n\in\Z}g_n\tau_n)$. Then 
$v_H$ is a valuation on $H$ such that, for any $g\in G$, $T_g=\{h\in H\mid v_H(h)=v_H(g)\}$. 
If we assimilate an element $\gamma$ of $\Gamma$ with $v_H^{-1}(\{\gamma\})$,
it follows that $v_H$ is the type-valuation associated to $\precsim_H$. 
Now  we extend 
$v_H$ to a valuation $v_F:F\to\Gamma\cup\{a,\infty\}$, where $a$ is a new element such that $a<\Gamma$, as follows: 
$v_F(k,h)=\begin{cases}
         a \text{ if }k\neq0.\\
        v_H(h)\text{ if }k=0.
        \end{cases}$\newline
If we assimilate elements of $\Gamma\cup\{a\}$ with their $v_F$-fiber, then 
$v_F$ is the type-valuation associated to $\precsim_F$.
  Now take $z:=(-1,\sum_{n\in\Z}(0,0)\tau_n)\in F$ and $f:=(0,\sum_{n\in\Z}g_n\tau_n)\in F$, where $g_0=(1,0)$ and $g_n=(0,0)$ for $n\neq0$. We have 
$v_F(f)=(0,1)$ but $v_F(z+g-z)=(-1,1)<v_F(f)$. In particular, $F^{(0,1)}=F_{(-1,2)}$ is not normal in $F$. This shows that the groups $G^{\gamma}$
and $G_{\gamma}$ of theorem \ref{structuretheorem} are not always normal in $G$.
 \end{Ex}

 We can also reformulate Theorem \ref{structuretheorem} in terms of C-relations:
 
 \begin{Thm}
  Let $(G,C)$ be a C-group. There exists a valuation $v:G\to\Gamma\cup\{\infty\}$ such that the following 
  holds:\begin{enumerate}
         \item For any $\gamma\in\Gamma$, $C$ induces a C-relation $C_{\gamma}$ on the quotient 
         $G^{\gamma}/G_{\gamma}$ defined by the formula 
         $C_{\gamma}(fG_{\gamma},gG_{\gamma},hG_{\gamma})\Leftrightarrow f\ih\notin G_{\gamma}\wedge(g\ih\in G_{\gamma}\vee C(f,g,h))$.
         \item For each $\gamma\in\Gamma$, $C_{\gamma}$ is of elementary type.
         \item If $\gamma\leq \delta$, if $C_{\gamma},C_{\delta}$ are both valuational, 
 then there exists $\alpha$ between $\gamma$ and $\delta$ such that
 $C_{\alpha}$ is order-type.
        \end{enumerate}

 \end{Thm}

\section{C-minimal groups}\label{Cminsection}

 We now want to interpret the results on C-minimal groups given in \cite{Macstein} in view of our structure theorem 
 \ref{structuretheorem}. Note that the C-relations considered in \cite{Macstein} are dense, i.e they satisfy the extra axioms: 
$x\neq y\Rightarrow\exists z, (z\neq y\wedge C(x,y,z))$ and $\exists x\exists y, y\neq x$.
 The authors of 
 \cite{Delon} and \cite{Adeleke2} described how to obtain the canonical tree associated to a given C-structure
 (see Proposition 1.5 in \cite{Delon} and Theorem 12.4 in \cite{Adeleke2}). 
 If $(M,C)$ is a C-structure, then we can define a partial quasi-order $\precsim$ on the set 
 $M^2$ by 
 $(x,y)\precsim (u,v)\Leftrightarrow\neg C(u,x,y)\wedge \neg C(v,x,y)$. We then
 define the canonical tree $(\T,\leq)$ of $(M,C)$ as the quotient $\T:=M^2/\sim$ endowed with the partial order 
 $\leq$ induced by $\precsim$. To simplify notations, we will refer to elements of $\T$ by one of their 
 representatives in $M^2$. Note that $(x,y)=(y,x)$ for any $x,y$. 
 
 If $(G,C)$ is a C-group with canonical tree $\T$, then we see that $G$ induces a right action on $\T$ by
 $(x,y).g:=(xg,yg)$. Note that the partial order on $\T$ is compatible with this action in the sense that 
 $(x,y)\leq (u,v)\Rightarrow(x,y).g\precsim (u,v).g$ (this follows directly from the fact that $C$ is compatible).
 The authors of \cite{Macstein} 
 described dense C-minimal groups by looking at the orbits of this action. They distinguished three cases:

 \begin{enumerate}
  \item All orbits are antichains.
  \item One orbit is a non-trivial chain. 
  \item No orbit is a non-trivial chain and there exists one non-trivial orbit which is not an antichain.
 \end{enumerate}

 Now let $\precsim$ be the C-q.o associated to $C$.
 We want to interpret this trichotomy in terms of $\precsim$.
More precisely, we want to see how the type of elements $x$ and $y$ influences 
the orbit of $(x,y)$. 
 Note that  the partial order $\leq$ of $\T$ is given by \newline
 $(x,y)\leq (u,v)\Leftrightarrow \{u\iy,v\iy\}\precsim x\iy$.
 We first want to describe the structure of the tree $\T$ in the order-type case: 

\begin{Lem}\label{chainordertype}
 Assume $(G,\precsim)$ is an order-type C-q.o.g and set 
 $\chain:=\{(x,y)\in\T\mid x\neq y\}$. Then $\chain$ is a non-trivial chain and an orbit under the action of $G$.
\end{Lem}

\begin{proof}
 Denote by $\leq$ the underlying order on $G$.
 Let $(x,y),(u,v)\in\chain$. Note that since $(x,y)=(y,x)$, we can assume 
 that $x<y$ and $u<v$. We have 
$(x,y)\leq(u,v)\Leftrightarrow \{u\iy,v\iy\}\precsim x\iy$. 
We saw in the proof of Proposition \ref{ordertypegroup} that $x<y$ is equivalent to $x\iy\in G^-$. Since 
$\precsim$ is trivial on $G^-$, $\{u\iy,v\iy\}\precsim x\iy$ is equivalent 
to $u\iy,v\iy\in G^-\cup\{1\}$. This in turn  is equivalent to $u\leq y\wedge v\leq y$.
  Since $u<v$, this is equivalent to $v\leq y$. Thus, we have 
$(x,y)\leq(u,v)\Leftrightarrow v\leq y$ and it follows that $\chain$ is a chain. Note that it also shows:\newline
$(\ast)$  $(u<v\wedge x<y\wedge y=v)\Rightarrow (x,y)=(u,v)$.\newline 
 Now we want to show that $(x,y)$ and $(u,v)$ are in the same orbit. 
Set $g:=\iy v$. Note that by definition of order-type C-relations in Example \ref{exempleCgroupe}(a), 
$<$ is compatible with the group operation, so we have  
$xg<yg$. Moreover, we have $u<v$ and $yg=v$. By $(\ast)$, this implies that we have $(x,y).g=(u,v)$.
\end{proof}

 \begin{Lem}\label{treequotient}
  Let $(G,\precsim)$ be a C-q.o.g (not necessarily minimal) and $g\in G$. Let 
  $(\T,\leq)$ be the canonical tree associated to $G^g$ and $(\T',\leq')$ the canonical tree 
  associated to $G^g/G_g$. If $x,y,u,v\in T_g$ are such that 
  $x\iy,u\iv\in T_g$, then 
  $(x,y)\leq(u,v)$ if and only if $(xG_g,yG_g)\leq'(uG_g,vG_g)$.
 \end{Lem}

 \begin{proof}
 By definition of the q.o on $G^g/G_g$ and since $x\iy\notin G_g$, we have \newline
  $\{u\iy,v\iy\}\precsim x\iy$ if and only if $\{u\iy G_g, v\iy G_g\}\precsim x\iy G_g$.
 \end{proof}

 \begin{Lem}\label{orbitotype}
  Let $(G,\precsim)$ be a C-q.o.g.
  Let $x\in G$ and $y\in G^x$. The following holds:
  \begin{enumerate}[(i)]
   \item If  $x\iy$ is v-type, then the orbit of $(x,y)$ under the action of $G$ is an antichain.
   \item If $x\iy\notin T_x$, then the orbit of $(x,y)$ under the action of $G^x$ is not a chain.
   \item The orbit of $(x,y)$ under the action of $G^x$ is a non-trivial chain if and only if $x$ is o-type and 
   $x\iy\in T_x$.
  \end{enumerate}

 \end{Lem}

 \begin{proof}
 \begin{enumerate}[(i)]
  \item Assume that $x\iy$ is v-type and let $g\in G$. We want to show that 
  $(x,y)$ and $(xg,yg)$ are either incomparable or equal. 
   Assume 
  $(xg,yg)\leq (x,y)$. This means $\{x\ig\iy, y\ig\iy\}\precsim x\iy$. Since $x\iy$ is v-type, 
  $y\ig\iy\precsim x\iy$ implies $yg\iy\precsim x\iy$ 
  (indeed, if $y\ig\iy$ is v-type, then $yg\iy\sim y\ig\iy$. If $y\ig\iy$ is o-type, then we have 
  $y\ig\iy\in G_{x\iy}$. Since $G_{x\iy}$ is a group, this implies $yg\iy\in G_{x\iy}$, hence $yg\iy\precsim x\iy$). 
  Moreover, if we conjugate the inequality $y\ig\iy\precsim x\iy$ by 
  $x\iy$, then we obtain $xg\ix\precsim x\iy\sim y\ix$. By $(CQ_2)$, $xg\ix\precsim y\ix$ implies 
  $xg\iy\precsim x\iy$. Thus, we have 
  $\{xg\iy,yg\iy\}\precsim x\iy$, which means $(x,y)\leq (xg,yg)$, so 
  $(x,y)$ and $(xg,yg)$ are equal. Now if we assume that $(x,y)\leq (x,y).g$ instead of $(x,y).g\leq (x,y)$ at the beginning, then by compatibility of the action we have 
$(x,y).\ig\leq (x,y)$, which brings us back to the previous case.
  \item Assume $x\iy\notin T_x$. Since $x\notin G_x$ and $x,\iy\in G^x$, it follows that $y\notin G_x$, hence $y\in T_x$. We thus have 
  $x\iy\precnsim \{x,\iy\}$.
  Taking $g:=\iy$, we cannot have $yg\iy\precsim x\iy$ and we also cannot have $x\ig\iy\precsim x\iy$. Therefore, neither 
  $(x,y)\leq(xg,yg)$ nor $(xg,yg)\leq(x,y)$ is true.
  \item If the orbit of $(x,y)$ under $G^x$ is a chain, then by (ii) we must have 
  $x\iy\in T_x$. By (i), $x$ cannot be v-type. Conversely, assume $x$ is o-type with 
  $x\iy\in T_x$. Since $G^x/G_x$ is order-type, and since $xG_g\neq yG_g$, it follows from Lemma \ref{chainordertype} that 
  the orbit of $(xG_g,yG_g)$ under the action of $G^g/G_g$ is a non-trivial chain. 
 It then follows from lemma \ref{treequotient}
  that the orbit of $(x,y)$ under $G^x$ is also a non-trivial chain.  
  \end{enumerate}
 \end{proof}

 \begin{Prop}\label{orbitclass}
  Let $(G,\precsim)$ be a C-q.o.g. The following holds:
  \begin{enumerate}[(i)]
       \item All orbits are antichains  if and only if every element if v-type.
       \item There exists an orbit which is a chain if and only if there exists 
       $g\in G$ o-type such that $T_g$ is maximal in the set of type-components of $G$ (for the order given in Proposition 
       \ref{ordertypecomponents}).
    \end{enumerate}
 \end{Prop}

 \begin{proof}
  If every orbit is an antichain, then by Lemma \ref{orbitotype}(iii) every element of $G$ must be v-type (otherwise we can 
  always choose $x,y\in G$ o-type with $x\iy\in T_x$, for example choose any $o^+$-type element $x$ and $y:=x^2$). The converse follows from 
  \ref{orbitotype}(i).  
  Now assume that $x\in G$ is an o-type element such that $T_g$ is maximal in the set of type-components of $G$.
  Take $y\in T_x$ with $x\iy\in T_x$.
  It follows from Lemma \ref{orbitotype}(iii) that the orbit of $(x,y)$ under $G$ is a chain. Conversely, assume there 
  is an orbit of an element $(x,y)$ which is a chain. Since $(x,y)=(y,x)$, we can assume without loss of generality that $y\precsim x$, hence $y\in G^x$.
 By Lemma \ref{orbitotype}(iii), this implies in particular that 
  $x$ is o-type with $x\iy\in T_x$. Assume that there is some $g\notin G^x$. We then have 
  $xg\iy,yg\iy,x\ig\iy,y\ig\iy\notin G^x$, so neither 
  $(x,y)\leq(xg,yg)$ nor $(xg,yg)\leq (x,y)$ can be true.
 \end{proof}

 We can now reformulate Theorems 4.4, 4.8 and 4.9 of \cite{Macstein} into the following result:
 
 \begin{Thm}\label{Cminimal}
 
   Let $(G,\precsim)$ be a C-minimal C-q.o.g and assume that $C$ is a dense C-relation. Then exactly one of the following holds:
   \begin{enumerate}[(i)]
    \item $\precsim$ comes from a valuation $v:G\to \Gamma\cup\{\infty\}$. In that case, we have the following:
  \begin{enumerate}[(1)]
   \item For any $\gamma\in\Gamma$, $G_{\gamma}$ and $G^{\gamma}$ are normal in $G$.
   \item The quotient 
  $G^{\gamma}/G_{\gamma}$ is abelian for all but finitely many $\gamma\in\Gamma$.
  \item If $G^{\gamma}/G_{\gamma}$ is infinite, then it is elementary abelian or divisible abelian. If it is divisible, then
  $G^{\gamma}$ is also abelian.
  \item There is a definable abelian subgroup $H$ of $G$ such that $G/H$ has finite exponent.
  \end{enumerate}
  \item  There exists an o-type element $g\in G$ such that $T_g$ is maximal in the set of type-components of $G$.
  In that case $G$ is abelian and divisible,
  $G_g$ is C-minimal and $G^g/G_g$ is o-minimal.
  \item $G$ contains o-type elements, but $T_g$ is never maximal for any $g$ o-type. In that case there exists 
  $g\in G$ such that the following holds:
  \begin{enumerate}[(1)]
  \item The final segment $\{h\in G\mid g\precsim h\}$ only contains v-type elements.
  \item There is a definable subgroup $H$ of $G$ such that $G/H$ has finite exponent.
  \end{enumerate}
  
   \end{enumerate}
 \end{Thm}
 
\begin{proof}
 Cases (i) and (ii) are direct reformulations of theorems 4.4 and 4.8 from \cite{Macstein} using our 
 Proposition \ref{orbitclass}. For (iii), we know from Proposition \ref{orbitclass} and from Theorem 4.9 of \cite{Macstein} that there exists 
 $w:=(g,1)\in \T$, $g\in G$, such that for any $t\leq w$, the orbit of $t$ under $G$ is an antichain. 
 Since $g.1^{-1}\in T_g$, and since the orbit of $w$ under $G$ is an antichain, Lemma \ref{orbitotype}(iii) implies
  that $g$ is v-type.  Now let 
  $u\in G$ with $g\precsim u$. We  have $\{1,g\}\precsim u$. By definition of $\leq$, this implies 
$(u,1)\leq (g,1)=w$, so the orbit of $(u,1)$ under $G$ is an antichain. Since moreover $u\in T_u$, it follows 
  from Lemma \ref{orbitotype}(iii) that $u$ is v-type.
\end{proof}

\begin{Rem}
\begin{enumerate}
 \item Theorem \ref{Cminimal} shows in particular that, if $G$ is C-minimal, then the set of type-components has a maximum.
 Thus, the ``ordered'' parts cannot 
 alternate indefinitely with the ``valued'' parts. Eventually, the group has to either stay valuational-like or 
 stay order-type-like.
 \item Theorem \ref{Cminimal} leaves open the question of welding in the case of C-minimality. More precisely, 
 we don't know if it is possible to have welding in case (ii).
  \item We needed the assumption of density in Theorem \ref{Cminimal} in order to apply the results of \cite{Macstein}. In a coming paper, we will 
go further and explore the structure of  general C-minimal groups, i.e without the density assumption.
\end{enumerate}

\end{Rem}

 \
\\

 {\footnotesize FACHBEREICH MATHEMATIK UND STATISTIK, 

UNIVERSITÄT KONSTANZ,

78457, GERMANY.\\

UNIVERSITÉ PARIS DIDEROT,

IMJ-PRG,

75013 PARIS, FRANCE.\\

\emph{Email address:} gabriel.lehericy@uni-konstanz.de}
 

\begin{thebibliography}{1}
 
 \bibitem{Adeleke1}  S.A. Adeleke and P.M. Neumann \emph{Primitive permutation groups with primitive Jordan sets},
  Journal of the London Mathematical Society 53(2), April 1996.
  \bibitem{Adeleke2} S.A. Adeleke and P.M. Neumann \emph{Relations related to betweenness: their 
  structure and their automorphisms},
  Memoirs of the American Mathematical Society 623, January 1998.
  \bibitem{Delon} Françoise Delon : \emph{C-minimal structures without the density assumption}, In
Raf Cluckers, Johannes Nicaise et Julien Sebag, éditeurs : Motivic Integration 
and its Interactions with Model Theory and Non-Archimedean Geometry.
Cambridge University Press, Berlin, 2011.
\bibitem{Fakhruddin}
                   Syed M.Fakhruddin,\emph{Quasi-ordered fields},
                   Journal of Pure and Applied Algebra 45, 207-210, 1987.
                   
  \bibitem{Lehericy} Gabriel Lehéricy, \emph{A structure theorem for abelian quasi-ordered groups},
  preprint, arXiv number 1606.07710v4, 2016, submitted.
                    \bibitem{Macstein}
                   Dugald Macpherson and Charles Steinhorn,\emph{On variants of o-minimality},
                   Annals of Pure and Applied Logic 79, 165-209, 1996.
                   
   \bibitem{Simonetta1} Patrick Simonetta, \emph{Abelian C-minimal groups}, Annals of pure 
   and applied logic, 110 (1-3):1-22, 2001.
   \bibitem{Simonetta2} Patrick Simonetta, \emph{On non-abelian C-minimal groups}, 
   Annals of pure and applied logic, 122, 263 – 287, 2003.
 \end{thebibliography}
\end{document}